\documentclass[11pt]{amsart}

\title{A trace map on higher scissors congruence groups}

\author[Bohmann]{Anna Marie Bohmann}
\address{Department of Mathematics, Vanderbilt University, Nashville, TN}
\email{am.bohmann@vanderbilt.edu}
\author[Gerhardt]{Teena Gerhardt}
\address{Department of Mathematics, Michigan State University, East Lansing, MI}
\email{teena@math.msu.edu}
\author[Malkiewich]{Cary Malkiewich}
\address{Department of Mathematics, Binghamton University, Binghamton, NY}
\email{malkiewich@math.binghamton.edu}
\author[Merling]{Mona Merling}
\address{Department of Mathematics, University of Pennsylvania, Philadelphia, PA}
\email{mmerling@math.upenn.edu}
\author[Zakharevich]{Inna Zakharevich}
\address{Department of Mathematics, Cornell University, Ithaca, NY}
\email{zakh@math.cornell.edu}

\usepackage{
amssymb,
color,
stmaryrd}
\usepackage{todonotes}
\usepackage{mathtools}

\usepackage{tikz-cd, enumerate}
\usetikzlibrary{matrix, arrows, decorations.markings, patterns, calc}

\usepackage[margin=1.25in]{geometry}
\usepackage{xcolor}
\definecolor{darkgreen}{rgb}{0,0.30,0} 
\definecolor{darkred}{rgb}{0.75,0,0}
\definecolor{darkblue}{rgb}{0,0,0.6} 
\usepackage[pdfborder=0,pagebackref,colorlinks,citecolor=darkgreen,linkcolor=darkgreen,urlcolor=darkblue]{hyperref}
\renewcommand*{\backref}[1]{}
\renewcommand*{\backrefalt}[4]{({%
    \ifcase #1 Not cited.%
          \or On p.~#2%
          \else On pp.~#2%
    \fi%
    })}

\def\makeautorefname#1#2{\expandafter\def\csname#1autorefname\endcsname{#2}}
\makeautorefname{equation}{Equation}%
\makeautorefname{footnote}{footnote}%
\makeautorefname{item}{item}%
\makeautorefname{figure}{Figure}%
\makeautorefname{table}{Table}%
\makeautorefname{part}{Part}%
\makeautorefname{appendix}{Appendix}%
\makeautorefname{chapter}{Chapter}%
\makeautorefname{section}{Section}%
\makeautorefname{subsection}{Section}%
\makeautorefname{subsubsection}{Section}%
\makeautorefname{paragraph}{Paragraph}%
\makeautorefname{subparagraph}{Paragraph}%
\makeautorefname{theorem}{Theorem}%
\makeautorefname{thm}{Theorem}%
\makeautorefname{addm}{Addendum}%
\makeautorefname{mainthm}{Main theorem}%
\makeautorefname{corollary}{Corollary}%
\makeautorefname{cor}{Corollary}%
\makeautorefname{lemma}{Lemma}%
\makeautorefname{lem}{Lemma}%
\makeautorefname{sublemma}{Sublemma}%
\makeautorefname{sublem}{Sublemma}%
\makeautorefname{subl}{Sublemma}%
\makeautorefname{prop}{Proposition}%
\makeautorefname{property}{Property}
\makeautorefname{pro}{Property}
\makeautorefname{sch}{Scholium}%
\makeautorefname{step}{Step}%
\makeautorefname{conject}{Conjecture}%
\makeautorefname{conj}{Conjecture}%
\makeautorefname{questn}{Question}
\makeautorefname{quest}{Question}
\makeautorefname{qn}{Question}
\makeautorefname{definition}{Definition}%
\makeautorefname{defn}{Definition}%
\makeautorefname{defi}{Definition}%
\makeautorefname{def}{Definition}%
\makeautorefname{dfn}{Definition}%
\makeautorefname{df}{Definition}%
\makeautorefname{notation}{Notation}
\makeautorefname{notn}{Notation}
\makeautorefname{rem}{Remark}%
\makeautorefname{rems}{Remarks}%
\makeautorefname{rmk}{Remark}%
\makeautorefname{rk}{Remark}%
\makeautorefname{remarks}{Remarks}%
\makeautorefname{rems}{Remarks}%
\makeautorefname{rmks}{Remarks}%
\makeautorefname{rks}{Remarks}%
\makeautorefname{example}{Example}%
\makeautorefname{examp}{Example}%
\makeautorefname{exmp}{Example}%
\makeautorefname{exam}{Example}%
\makeautorefname{exa}{Example}%
\makeautorefname{ex}{Example}%
\makeautorefname{axiom}{Axiom}%
\makeautorefname{axi}{Axiom}%
\makeautorefname{ax}{Axiom}%
\makeautorefname{case}{Case}%
\makeautorefname{claim}{Claim}%
\makeautorefname{clm}{Claim}%
\makeautorefname{assumpt}{Assumption}%
\makeautorefname{asses}{Assumptions}%
\makeautorefname{conclusion}{Conclusion}%
\makeautorefname{concl}{Conclusion}%
\makeautorefname{conc}{Conclusion}%
\makeautorefname{cond}{Condition}%
\makeautorefname{const}{Construction}%
\makeautorefname{con}{Construction}%
\makeautorefname{criterion}{Criterion}%
\makeautorefname{criter}{Criterion}%
\makeautorefname{crit}{Criterion}%
\makeautorefname{exercise}{Exercise}%
\makeautorefname{exer}{Exercise}%
\makeautorefname{exe}{Exercise}%
\makeautorefname{warn}{Warning}%
\newtheorem{thm}{Theorem}[section]
\newtheorem{cor}{Corollary}[section]
\newtheorem{lem}{Lemma}[section]
\newtheorem{prop}{Proposition}[section]
\theoremstyle{definition}
\newtheorem{df}{Definition}[section]

\newtheorem{notn}{Notation}[section]

\theoremstyle{remark}
\newtheorem{ex}{Example}[section]
\newtheorem{rmk}{Remark}[section]
\numberwithin{equation}{section}
\numberwithin{figure}{section}

\makeatletter
\let\c@cor=\c@thm
\let\c@prop=\c@thm
\let\c@lem=\c@thm
\let\c@df=\c@thm
\let\c@ex=\c@thm
\let\c@warn=\c@thm
\let\c@rmk=\c@thm
\let\c@notn=\c@thm
\let\c@equation\c@thm
\let\c@figure\c@thm
\let\c@table\c@thm
\makeatother

\hypersetup{linktoc=all}

\newcommand{\Z}{\mathbb{Z}}
\newcommand{\Q}{\mathbb{Q}}
\newcommand{\R}{\mathbb{R}}

\newcommand{\sma}{\wedge}

\newcommand{\C}{\mathcal{C}}
\newcommand{\D}{\mathcal{D}}
\newcommand{\F}{\mathcal{F}}
\newcommand{\E}{\mathcal{E}}

\newcommand{\W}{\mathcal{W}}
\newcommand{\U}{\mathcal{U}}
\newcommand{\Set}{\mathbf{Set}}

\renewcommand{\O}{\mathcal{O}}

\DeclareMathOperator{\ob}{ob}

\DeclareMathOperator{\colim}{colim}

\DeclareMathOperator{\Hom}{Hom}
\DeclareMathOperator{\diag}{\operatorname{diag}}

\newcommand{\id}{\mathrm{id}}
\newcommand{\inv}{{-1}}
\newcommand{\initial}{\varnothing}
\newcommand{\Pol}[2]{\mathcal P^{#1}_{\hspace{.1em} #2}}
\newcommand{\tr}{\textrm{tr}}
\newcommand{\Cat}{\mathbf{Cat}}
\newcommand{\pCat}{{\mathbf{Cat}_{\mathrm{pt}}}}
\newcommand{\Catfam}{\mathbf{CatFam}}
\newcommand{\Spc}{\mathbf{Spc}}
\newcommand{\Sp}{\mathbf{Sp}}

\newcommand{\qqand}{\qquad\hbox{and}\qquad}

\newcommand{\asm}{\mathrm{Asm}}

\newcommand{\sd}{\operatorname{sd}}
\newcommand{\ab}{\mathrm{ab}}

\newcommand{\catcov}{category with covering families}
\newcommand{\catcovs}{categories with covering families}

\newcommand{\Catcovs}{Categories with covering families}

\tikzset{
covsubmap/.style={densely dotted,thick,
      decoration={markings,
                    mark=at position 1 with {\arrow[line width=0.1mm, scale= 1.2]{>}} 
}, postaction={decorate}
}
}

\tikzset{
movemapstyle/.style={>->}
}

\newcommand{\csm}{\begin{tikzpicture}
\draw[covsubmap] (0,0)--(1,0);
\end{tikzpicture}}
\newcommand{\movemap}{\begin{tikzpicture}\draw[>->] (0,0)--(1,0);
\end{tikzpicture}}
\newcommand{\csmrev}{\begin{tikzpicture}
\draw[covsubmap] (1,0)--(0,0);
\end{tikzpicture}}
\newcommand{\movemaprev}{\begin{tikzpicture}\draw[>->] (1,0)--(0,0);
\end{tikzpicture}}

\tikzcdset{arrows={no head}}

\newcommand{\sbt}{\,\begin{picture}(-1,1)(0.5,-1)\circle*{1.8}\end{picture}\hspace{.05cm}}

\newcommand{\sbtup}{\,\, \begin{picture}(-1,1)(0.5,-2)\circle*{1.8}\end{picture}\hspace{.05cm}}

\usepackage{xcolor}
\definecolor{seagreen}{RGB}{46,139,87}
\definecolor{maroon}{RGB}{128,0,0}
\definecolor{darkviolet}{RGB}{148,0,211}

\newlength{\storeparskip}
\begin{document}

\begin{abstract}
Cut-and-paste $K$-theory has recently emerged as an important variant of higher algebraic $K$-theory. However, many of the powerful tools used to study classical higher algebraic $K$-theory do not yet have analogues in the cut-and-paste setting. In particular, there does not yet exist a sensible notion of the Dennis trace for cut-and-paste $K$-theory.

In this paper we address the particular case of the $K$-theory of polyhedra, also called scissors congruence $K$-theory. We introduce an explicit, computable trace map from the higher scissors congruence groups to group homology, and use this trace to prove the existence of some nonzero classes in the higher scissors congruence groups.

We also show that the $K$-theory of polyhedra is a homotopy orbit spectrum. This fits into Thomason's general framework of $K$-theory commuting with homotopy colimits, but we give a self-contained proof. We then use this result to re-interpret the trace map as a partial inverse to the map that commutes homotopy orbits with algebraic $K$-theory. 
\end{abstract}





\maketitle

\begingroup%
\setcounter{secnumdepth}{1}
\setcounter{tocdepth}{1}
\tableofcontents
\endgroup%

\section{Introduction}

The classical scissors congruence problem asks: given two polyhedra $P$ and $Q$ in Euclidean space $E^n$, can one cut $P$ into a finite number of smaller polyhedra, and reassemble these to form $Q$?
In the Euclidean plane $E^2$, area is the only scissors congruence invariant, meaning that this is possible as long as $P$ and $Q$ have the same area. Hilbert's Third Problem famously asks the analoguous question in  3-dimensional Euclidean space $E^3$: Are any two polyhedra $P$ and $Q$ with the same volume scissors congruent? The answer to Hilbert's problem is negative---one needs two invariants, the volume and the Dehn invariant, to determine if two polyhedra in $E^3$ are scissors congruent \cite{dehn,sydler}.  In fact, Dehn shows that a cube and a regular tetrahedron of the same volume are not scissors congruent.

The generalized Hilbert's Third Problem for any standard geometry $X$ (Euclidean, spherical, or hyperbolic) asks whether volume and generalized Dehn invariants are the only scissors congruence invariants in higher dimensions. In the Euclidean case, for instance, Jessen shows that this is true in dimension 4 \cite{jessen}, but the question in dimensions 5 and above remains open. 

Modern approaches to Hilbert's Third Problem are typically algebraic. The scissors congruence group $\mathcal P(X, G)$ is defined as the free abelian group on the polytopes in $X$, modulo cutting and pasting, and modulo translation of the pieces by elements of a group $G$ of isometries of $X$ \cite{sah_79,Dupont-Book}. The volume and generalized Dehn invariants are then homomorphisms out of the scissors congruence group $\mathcal P(X, G)$.

In \cite{inna-assem}, Zakharevich shows that $\mathcal P(X, G)$ arises as the zeroth homotopy theory group of a $K$-theory spectrum $K(\Pol{X}{G})$, where $\Pol{X}{G}$ is a category  defined using polytopes in $X$ and isometries in $G$, recalled in \autoref{polytopeex}.  This makes the study of scissors congruence into a branch of algebraic $K$-theory, and produces a natural extension, a sequence of higher scissors congruence groups $K_n(\Pol{X}{G})$. These represent the higher additive invariants of scissors congruence, which encode not only whether two polyhedra are scissors congruent, but also how they are congruent.  The higher scissors congruence groups also have surprising connections to the rational $K$-theory of fields, see e.g.~\cite{cz-hilbert,scissors_thom}.

We emphasize that scissors congruence $K$-theory is not the $K$-theory of an exact category or Waldhausen category, and thus is not defined by breaking up exact sequences, as in these classical cases. It is a more recent variant called ``cut-and-paste'' $K$-theory, or combinatorial $K$-theory. Other notable forms of cut-and-paste $K$-theory include the $K$-theory of varieties, which arises in the study of motives in algebraic geometry \cite{nicaise_sebag_2011,rondigs2016grothendieck,inna-lefschetz,campbell-varieties,cwz}, and the $K$-theory of $n$-dimensional manifolds \cite{cut_and_paste_wit}.

Cut-and-paste $K$-theory opens up new possibilities, and also new challenges. While many classical $K$-theoretic tools carry over to this new context (see e.g.~\cite{inna-assem,cz-cgw}), many open questions remain, particularly with regard to incorporating trace methods into cut-and-paste $K$-theory. In the trace method approach to classical algebraic $K$-theory, higher algebraic $K$-theory groups are approximated via trace maps to more computable invariants, such as Hochschild homology ($\mathrm{HH}$), topological Hochschild homology ($\mathrm{THH}$), and topological cyclic homology ($\mathrm{TC}$). In the past 30 years trace methods have been enormously successful in facilitating the calculation of algebraic $K$-theory groups. Thus, one would like to import these techniques to the study of cut-and-paste $K$-theory as well. However, it is not yet clear how to generalize the classical trace maps, such as the Dennis trace map $K(R) \to \mathrm{THH}(R)$, to this setting.

In this paper we construct an explicit, computable trace map from the higher $K$-theory of polytopes to group homology. Our techniques apply more generally to any variant of $K$-theory that comes from an assembler (as defined in \cite{inna-assem}) or, more generally still, from a {\catcov} (\autoref{family_structure} below). We first define the notion of a homotopy orbit {\catcov} $\C_{hG}$ arising from any $G$-{\catcov} $\C$. We prove:
\begin{thm}[Commutation with homotopy orbits]\label{mainthm}
For every group $G$, and every $G$-{\catcov} $\C$, there is an equivalence of spectra
\[  K(\C)_{hG} \xrightarrow{\ \asm\ }  K(\C_{hG}). \]
\end{thm}

This is a variant of a theorem of Thomason \cite{thomason1, thomason2}, stating that $K$-theory commutes with homotopy orbits in the setting of permutative categories. The above result corresponds to a variant of Thomason's theorem that applies to the $K$-theory of general multicategories, as defined in \cite{EM_multi, JY, Nikolaus14}, rather than to those that come from permutative categories. However, in lieu of deducing the result directly from Thomason's, we give a new self-contained proof that applies directly to {\catcovs}.

The significance of this result to polytopes is as follows. The category $\Pol{X}{G}$ used to define scissors congruence $K$-theory is the homotopy orbits of the category $\Pol{X}{1}$, in which the chosen group of isometries is the trivial group $G = 1$.

\begin{cor} \label{polyhedrathrm}
	Scissors congruence $K$-theory is a homotopy orbit spectrum,
\begin{equation*} K(\Pol{X}{1})_{hG} \overset\sim\longrightarrow K(\Pol{X}{G}). \end{equation*} 
\end{cor}

This is a surprising conclusion, with significant computational consequences. In  \cite{scissors_thom} the third author uses this result to describe the homotopy type of $K(\Pol{X}{G})$ as a Thom spectrum, giving computations of the higher $K$-groups for each of the one-dimensional geometries.

Returning to the case that $\C$ is any $G$-{\catcov}, let $A$ be any $\Z[G]$-module, and $\mu\colon K_0(\C) \to A$ any $G$-equivariant measure on $\C$ with values in $A$ (see \autoref{volume}).
\begin{thm}[Trace]\label{intro_trace}
There is a natural, explicitly defined trace map
\[ K(\C_{hG}) \xrightarrow{\ \tr\ } (HA)_{hG},  \]
which produces a map to group homology 
\[ K_n(\C_{hG})\xrightarrow{\ \tr\ } H_n(G;A).\]
Moreover, the composite of the trace map with the equivalence of \autoref{mainthm}
\[ K(\C)_{hG} \xrightarrow{\ \asm\ } K(\C_{hG}) \xrightarrow{\ \tr\ } (HA)_{hG}  \]
is the homotopy $G$-orbits of the map $K(\C) \to HA$ induced by the measure $\mu$.
\end{thm}

In particular, taking $\mu$ to be the measure on polytopes in $X$ given by volume, we get a trace map on the higher scissors congruence groups,
\[ K_n(\Pol{X}{G})\xrightarrow{\ \tr\ } H_n(G;\R). \]
In \autoref{rotation_example} we describe a few explicit classes in the higher scissors congruence groups, and use this trace to prove they are nontrivial.

On the other hand, if we take take $A=K_0(\Pol{X}{1})$ to be the scissors congruence group with trivial isometries and take $\mu$ to be the identity map $K_0(\Pol{X}{1}) \to K_0(\Pol{X}{1})$,  results of \cite{scissors_thom} show that $K(\Pol{X}{1}) \to HA$ is a rational equivalence.  Hence the trace is rationally an inverse to the assembly map. In other words, it is powerful enough to detect all of the rational higher scissors congruence groups (see \autoref{example:caryscoolpaperpolyhedra}). The trace of \autoref{intro_trace} is therefore an essential tool for our computations, and we expect to deploy it further in future work.

\begin{rmk}
	The equivalence of \autoref{mainthm} is reminiscient of the assembly map in the $K$-theory of group rings,
\begin{equation}\label{intro_fj_assembly}
	 K(R)_{hG} \longrightarrow  K(R[G]).
\end{equation}
This is the source of the notation ``$\asm$'' in \autoref{mainthm}. The assembly map for rings is not an equivalence in general, but is conjectured to be an equivalence when $R$ is a regular ring and $G$ is torsion-free. This is the famous Farrell--Jones conjecture, see e.g.~\cite{assembly_survey,lrrv}.
	
	One must be careful because the notion of homotopy orbits that we use here for {\catcovs}, or more generally for multicategories, does not agree with homotopy orbits in the setting of idempotent-complete stable infinity categories. In that setting, the homotopy orbits of the category $\mathrm{Perf}(R)$ of perfect $R$-modules with respect to the trivial action of $G$ is computed as $\mathrm{Perf}(R)_{hG}\simeq \mathrm{Perf}(R[G])$, see \cite[Example 2.19]{CMNN} or \cite[Example 1.4]{bunke_cisinski_kasprowski_winges}. Thus, in that context, the classical Farrell--Jones conjecture is precisely about when $K$-theory commutes with homotopy orbits. This is in contrast to the setting of permutative categories or multicategories, where the two constructions always commute.

	We also remark that the compatibility of the trace map and the homotopy orbits in \autoref{intro_trace} parallels the ring case, where the Dennis trace
	\begin{equation*}\label{intro_dennis_trace}
          	 K(R[G]) \longrightarrow  \mathit{TC}(R[G]) \longrightarrow \mathrm{THH}(R[G]) 
	\end{equation*}
	is used in some cases to partially invert the assembly map, see e.g.~\cite{lrrv}.
\end{rmk}

\subsection{Conventions and notations}
 We let $\Spc$ denote the category of spaces, assumed to be compactly generated
weak Hausdorff.  We let $\Sp$ denote the category of spectra (referred to as ``prespectra''  in \cite{mandell_may_shipley_schwede}), whose levels are
assumed to be compactly generated weak Hausdorff as well.  $\Cat$ denotes the
category of (small) categories, and $\pCat$ 
 denotes the category of (small) categories with a distinguished ``base'' object, which need not be either initial or terminal.  
\subsection{Acknowledgments} 
The authors would like to thank Mike Mandell and Kate Ponto for useful conversations at the beginning of this work.  We thank Tony Elmendorf and Mike Mandell for patiently explaining to us their work on multicategories and Thomason's work on homotopy colimits of permutative categories. Merling also thanks George Raptis and Julia Semikina for valuable discussions.

The authors were partially supported by NSF grants for the Focused Research Group project Trace Methods and Applications for Cut-and-Paste $K$-theory, under the grant numbers DMS-2052849 (Bohmann), DMS-2052042 (Gerhardt), DMS-2052923 (Malkiewich), DMS-2052988 (Merling), and DMS-2052977 (Zakharevich). Additionally, Bohmann was partially supported by NSF grant DMS-2104300, Gerhardt was partially supported by NSF grant DMS-2104233,  Malkiewich was partially supported by NSF grant DMS-2005524,  Merling was partially supported by NSF grant DMS-1943925, and Zakharevich was partially supported by NSF grant  DMS-1846767. This material is in part based on work supported by the National Science Foundation under DMS-1928930, while the first, second and fourth authors were in residence at the Simons Laufer Mathematical Sciences Institute (previously known as MSRI) in Berkeley, California, during the Fall 2022 semester.

\section{{\Catcovs} and scissors congruence $K$-theory}

In this section we define the notion of a {\catcov} and explain how to
construct its $K$-theory spectrum.  A central example is scissors congruence
$K$-theory, which was defined by \cite{inna-assem} as the $K$-theory of an
``assembler.''   {\Catcovs} generalize assemblers, and our definitions here recover the scissors congruence $K$-theory of \cite{inna-assem}.  An advantage of this more general framework is that we recover interesting examples that do not arise from assemblers, including \autoref{ex:chaotic-group} below.

\begin{df}\label{family_structure}
  Let $\C$ be a category. A \emph{multi-morphism} in $\C$ is an object $B \in \C$ and a finite (possibly empty) set of maps in $\C$ with target $B$:
  \[ \{ f_i\colon A_i \to B \}_{i \in I}. \]
  A \emph{covering family structure} on $\C$ is a collection of multi-morphisms, called the \emph{covering families}, such that:
  \begin{itemize}
  	\item Each singleton containing an identity map $\{ A = A \}$ is a covering family.
  	\item Given a covering family
  	\[ \{ g_j\colon B_j \to C\}_{j \in J} \]
  	and for each $j \in J$ a covering family
  	\[ \{ f_{ij}\colon A_{ij} \to B_j \}_{i \in I_j}, \]
  	the collection of all composites $g_j \circ f_{ij}$ again forms a covering family:
  	\[ \{ g_j \circ f_{ij}\colon A_{ij} \to C \}_{j \in J, \ i \in I_j}. \]
      \end{itemize}

      A \emph{\catcov} is a small category $\C$, a covering family
      structure on $\C$, and a distinguished ``basepoint'' object $*\in \C$, such that
      \begin{itemize}
      \item[(B)] $\Hom(*,*) = \{1_*\}$ and $\Hom(c,*) = \initial$ for $c \neq *$.
      \item[(C)] For every finite (possibly empty) set $I$, the
        family $\{* \to *\}_{i\in I}$ is a covering family.
      \end{itemize}
      For a {\catcov} $\C$, we denote by $\C^\circ$ the full subcategory containing
      the non-basepoint objects of $\C$.
      
      A morphism of {\catcovs} is a functor $F\colon \C \to \D$ that preserves
      covering families and the basepoint. The category of {\catcovs} is denoted $\Catfam$.
\end{df}

\begin{rmk}
	The definition of a {\catcov} has an alternate phrasing that is more concise. The multi-morphisms of \autoref{family_structure} define a symmetric multicategory $\F\C$. A covering family structure on $\C$ is the same thing as a sub-multicategory $\widetilde{\C} \subseteq \F\C$.
\end{rmk}

\begin{rmk}\label{add_basepoint}
  If $\C$ has a covering family structure but no basepoint, we can add a basepoint in two different ways:
	\begin{itemize}
		\item A disjoint basepoint. We form the category $\C_+$ by adding  an object $*$ with
  \[\Hom(*,*) = \{1_*\} \qqand \Hom(*,c) = \Hom(c,*) = \initial \quad \forall
    c \neq *.\]
    		This becomes a {\catcov} if we insist that all families of the form $\{* \to *\}_{i\in I}$ are covering families.
    		\item An initial basepoint. We form the category $\C_*$ by adding an object $*$ and defining $\Hom(*,c)$ to be a one-point set for all $c$ and $\Hom(c,*)=\initial$ for all $c\neq *$.

    		For $c\in \C$, we declare that a collection of maps into $c$ is a covering family in $\C_*$ iff removing all the maps of the form $* \to c$ gives a covering family in $\C$. In particular, $* \to c$ is a cover in $\C_*$ iff $c$ admits an empty cover in $\C$.  If we also declare that families of the form $\{\ast\to \ast\}_{i\in I}$ are covering families, this construction yields a {\catcov} as well.
	\end{itemize}
\end{rmk}

There are many natural examples of \catcovs.

\begin{ex}
  Every Grothendieck site $\C$ is endowed with a covering family
  structure by taking the covering families to be all of the finite covers in the Grothendieck topology. This makes $\C_+$ and $\C_*$ into categories with covering families. This example is the source of the terminology.
\end{ex}

\begin{ex}\label{ex:groupascatwithoneobj}
  Let $G$ be a discrete group, considered as a category with one object. This
  has a covering family structure with covering families given by singleton morphisms $g$.
  Then $G_+$ and $G_*$ are categories with covering families.
\end{ex}

\begin{ex} \label{ex:chaotic-group}
  Let $A$ be a discrete abelian group, and write $\E\!A$ for the category whose
  set of objects is $A$ and in which there is a unique morphism between any two
  objects\footnote{This is sometimes also called the indiscrete or the chaotic
    category on the group $A$. Its classifying space is a model for the space
    $E\!A$.}.  We declare that $\{ a_i \to a \}_{i \in I}$ is a covering family if
  $\sum_{i\in I} a_i = a$. 
  In particular, the empty cover is a covering family of $0 \in A$. 
   It is tempting to declare this to be a {\catcov} with basepoint equal to the $0$
  element in $A$, but this will not work because $\Hom(a,0) \neq \initial$ for
  nonzero $a$.  Instead, we will use $\E\!A_+$ or $\E\!A_*$.
\end{ex}

\begin{ex} \label{ex:ordered-group} Let $G$ be a partially ordered discrete
  abelian group, considered as a category $\O_G$ under the partial order.  Then
  the full subcategory of elements $x \geq 0$ can be considered a {\catcov}
  $(\O_G)_{\geq 0}$ by declaring a family $\{a_i\to b\}_{i\in I}$ to be a
  covering family if $\sum_I a_i = b$.  With this definition, $(\O_G)_{\geq 0}$
  is a {\catcov} with basepoint equal to $0$.
\end{ex}

\begin{ex}\label{ex:assembler}
  An \emph{assembler} $\C$, originally defined in \cite{inna-assem},  is a
  Grothendieck site satisfying additional axioms. In particular, it has an
  initial object $\initial$, and two morphisms $A_i \to B$ and $A_j \to B$ are
  said to be disjoint if their pullback is $\initial$. Each assembler $\C$ can
  be considered a \catcov, using the
  pairwise-disjoint finite covering families of the Grothendieck structure.
  The initial object of the assembler is required to have the empty set as a
  covering family, and thus this will be a {\catcov} whose basepoint is the initial object $\initial$.  
\end{ex}

The motivating example of this paper is the category of polytopes. This forms an assembler, and therefore also a \catcov.
\begin{ex}\label{polytopeex}
  Let $X$ be any standard $n$-dimensional geometry (Euclidean, hyperbolic, or
  spherical), and let $G$ be any subgroup of the isometry group of $X$.  A polytope $P$ in $X$ is defined to be a finite union of nondegenerate $n$-simplices in $X$. Define $\Pol{X}{G}$ to be the category whose objects are the nonempty polytopes in $X$.   A morphism  $P \to Q$  in $\Pol{X}{G}$ consists of a choice of
  $g \in G$ such that $gP \subseteq Q$. In other words, morphisms in $\Pol{X}{G}$ allow us to move a
  polytope by any isometry in the subgroup $G$ and then include into another polytope.
 We define a covering family structure on $\Pol{X}{G}$ by declaring a collection of morphisms $\{ g_i\colon P_i \to P\}_{i\in I}$ to be a covering family if the polytopes $g_iP_i$ cover $P$ and the pairwise overlaps
  $g_iP_i \cap g_jP_j$ have measure zero. Equivalently, the union of all the $g_iP_i$ is equal to $P$ and each pair $g_iP_i$ and
  $g_jP_j$ has disjoint interiors. See also \cite[Example 3.6]{inna-assem}.
\end{ex}

We conclude the discussion of the definition with an important observation about
morphisms of \catcovs.

\begin{ex} \label{ex:collapse}
  There is a \emph{trivial \catcov}, $*$, whose underlying category is the
  trivial category with one object and only an identity morphism.  The family
  structure is the one enforced by Axiom (C).  Every {\catcov} $\C$ has a unique
  morphism $\C \to *$.
\end{ex}

Given a {\catcov} $\C$ and a pointed set $X$, we define a new {\catcov} $X\sma\C$ and observe that this construction is functorial in $X$.

  Let $X$ be a pointed set.  We can consider $X$ as a category by defining morphism sets $\Hom(a,b)$ to be one-point sets when $a=b$ or when $a=*$ is the basepoint and defining all other morphism sets to be empty.  That is, we take the discrete category on the non-basepoint elements of $X$ and then force the basepoint to be initial.

\begin{df} \label{wedgeex}
  Suppose $\C$ is a \catcov, and $X$ is a pointed set. We define a {\catcov} $X
  \wedge \C$ by $X\sma \C=\bigvee_{X^\circ}\C$, where $X^\circ$ denotes the non-basepoint elements of $X$.  In other words, 
  \[\ob (X \sma \C) = (\ob X \times \ob \C) / (\ob X \vee \ob \C),\]
  with morphisms induced from the morphisms in $X \times \C$, where the category
  structure on $X$ is as above.  The basepoint object in $X\sma C$ is the wedge point.  The full subcategory of
  non-basepoint objects is of the form $\coprod_{x\in X^\circ} \C^\circ$. Aside from the maps out of the basepoint, the  morphisms are internal to each copy of $\C^\circ$.

  The family structure is given by declaring $\{A_i \to B\}_{i\in I}$ to be a
  covering family if all $A_i$ and $B$ are contained in a single copy of $\C$,
  and it is a covering family in $\C$.  (The basepoint is considered to be in
  all copies of $\C$.)

   A map  of based sets $f\colon X \to Y$ induces a functor
   \[ X \sma \C \to Y \sma \C \] by applying $f$ to the indexing labels. This
   preserves covering family structure: the main thing to check is that when $f$
   takes an element $x\in X^\circ$ to the basepoint in $Y$, the induced functor
   $\C \to \{*\}$ preserves covers.  This holds, as noted in
   \autoref{ex:collapse}. We therefore get a functor from based sets to
   \catcovs
  \[ (-) \sma \C\colon \Set_* \to \Catfam. \]
\end{df}

As a step towards  constructing the algebraic $K$-theory of a {\catcov} $\C$, we define a ``category of covers'' associated to $\C$.  This generalizes the construction of the category of covers of an assembler, used in \cite{inna-assem} to construct the $K$-theory of an assembler.  See \autoref{rmk:catcoverforassemblers} below for this example.

\begin{df}\label{category_of_covers}
  For a {\catcov} $\C$, we define the \emph{category of covers}
  $\W(\C)$ as follows:
\begin{itemize}
\item An object is a finite tuple $\{A_i\}_{i\in I}$ of objects in
  $\C^\circ$. We allow $I$ to be empty.
  \item A morphism $\{A_i\}_{i\in I} \to \{B_j\}_{j\in J}$ is a
    morphism $f\colon I \to J$ of finite sets and covering families in $\C$
    \[ \{ f_i\colon A_i \to B_j \}_{i\in f^{-1}(j)} \]
    for each $j\in J$. In other words, the $A_i$ are partitioned into subsets, and each subset covers one of the $B_j$.
  \item We compose morphisms by composing the maps of finite sets and then composing the covers as in \autoref{family_structure}. In other words, we use the composition in the multicategory $\F\C$. 
  \end{itemize}
\end{df}

\begin{rmk}\label{basepoint_issue}
	The category $\W(\C)$ has a natural basepoint object, corresponding to $I = \initial$. It may not be a disjoint basepoint, since $\C$ could have objects that are covered by the empty family, as in \autoref{ex:chaotic-group}.
	
	Note also that if $\C$ lacks a basepoint and we form $\C_+$ and $\C_*$ as in \autoref{add_basepoint}, the resulting categories $\W(\C_+)$ and $\W(\C_*)$ are canonically isomorphic. In other words, maps of the form $* \to c$ do not affect the definition of $\W(\C)$.
      \end{rmk}

\begin{df}
  The category $\pCat$ has as objects small categories with a distinguished
  object $*$.  This object is not required to be either initial or terminal.
 \end{df}

\begin{rmk}
The fact that the distinguished object is not required to be either initial or terminal is significant because the nerve of a category with initial or terminal object is contractible.  In \autoref{ktheory} we define $K$-theory by taking nerves of categories of covers, which we now show are categories in $\pCat$.
\end{rmk}

\begin{lem}\label{w_functor}
	The category of covers construction $\W(-)$ forms a functor 
	\[ \W(-)\colon \Catfam \to \pCat. \]
\end{lem}

\begin{proof}
	For each functor $F\colon \C \to \D$ preserving covering families and
        the distinguished basepoint, we define $\W(\C) \to \W(\D)$ by applying
        $F$ to the objects $A_i$. This is well defined because $F$ preserves covering families.
	
	It is clear that this operation respects composition and identity maps in $\W(\C)$, so that $\W(F)$ is a morphism in $\pCat$.  It also respects compositions and identities in the $F$ variable so that $\W(-)$ is a functor $\Catfam\to\pCat$.
\end{proof}

We now define the algebraic $K$-theory of a category with covering families.
 \begin{df}\label{ktheory}
   Let $\C$ be a category with covering families. By combining \autoref{w_functor} and \autoref{wedgeex},  we get a functor from finite based sets to based spaces
\[ X \mapsto \left| N_{\sbt} \W\left( X \sma \C \right) \right|. \]
By definition, this is a $\Gamma$-space \cite{segal,bousfield_friedlander}. The \emph{$K$-theory spectrum} $K(\C)$ is the symmetric spectrum associated to this $\Gamma$-space. Concretely, at spectrum level $k$, it is the realization of the multisimplicial set
\begin{equation*}\label{sck_defn}
	[p,q_1,\dots,q_k] \mapsto N_p \W\big( S^1_{q_1} \sma \dotsb \sma S^1_{q_k} \sma \C \big),
\end{equation*}
where $S^1_{\sbt}$ denotes the simplicial circle. Taking the diagonal, this is also the realization of the simplicial set
\begin{equation*}
	p \mapsto N_p \W\big( S^k_p \sma \C \big),
\end{equation*}
where $S^k_{\sbt} = (S^1_{\sbt})^{\sma k}$ is the simplicial $k$-sphere. The bonding maps of the spectrum arise from the identifications of finite based sets
\[  S^1_{q_1} \sma \dotsb \sma S^1_{q_k} \cong S^1_{q_1} \sma \dotsb \sma S^1_{q_k} \sma S^1_1. \]
\end{df}

It is a consequence of \autoref{weak_additivity} below that this $\Gamma$-space is special. It follows that this spectrum is a positive $\Omega$-spectrum and is therefore semistable (see e.g.~\cite{schwede_symmetric_spectra}, \cite[Prop 5.6.4]{hovey_shipley_smith}). We are therefore free to ignore the symmetric spectrum structure and treat it as an ordinary sequential spectrum.

\begin{rmk}\label{rmk:catcoverforassemblers}
	When $\C$ is an assembler, as in \autoref{ex:assembler}, \autoref{ktheory} agrees with the $K$-theory of assemblers defined in \cite[Def 2.12]{inna-assem}.
\end{rmk}

\begin{ex}\label{polytope_spectrum}
  The $K$-theory of the category $\Pol{X}{G}$ from
  \autoref{polytopeex} is the \emph{scissors congruence $K$-theory spectrum}
  $K(\Pol{X}{G})$ (see \cite{inna-assem}). Its lowest homotopy group $K_0(\Pol{X}{G})$ is the classical
  scissors congruence group $\mathcal P(X,G)$ from
  e.g.~\cite{dupont_book,sah_79}.  
\end{ex}
      
\begin{ex}\label{KtheoryBG}
  Let $G$ be a discrete group. As in \autoref{ex:groupascatwithoneobj}, consider $G$ as a category with one object, and add a basepoint using \autoref{add_basepoint}. In fact, it does not matter whether we take $G_+$ or $G_*$; they give isomorphic $K$-theory spectra by \autoref{basepoint_issue}.
  
  We claim the resulting spectrum is stably equivalent to $\Sigma^\infty_+ BG$. 
    To see this, we observe that the category $\W(X \sma G_*)$ is exactly the same as $\W(X \sma
  \mathbb{S}_G)$, where $\mathbb{S}_G$ is the assembler in \cite[Example
  3.2]{inna-assem}. As observed there, the
  $\Gamma$-space $\left|N_{\sbt}\W(X\sma \mathbb{S}_G)\right|$ associated to the
  assembler $\mathbb{S}_G$, or equivalently, to the {\catcov} $G_\ast$,  defines the spectrum $\Sigma^\infty_+ BG$ by a version of the Barratt--Priddy--Quillen--Segal Theorem. 
\end{ex}

\begin{ex}
  Let $A$ be a discrete abelian group, and consider the category with covering families
  $\E\! A_*$ from Example~\ref{ex:chaotic-group}. Then there is an equivalence
  $K(\E\! A_*) \simeq HA$ where $HA$ denotes the Eilenberg--MacLane spectrum on
  $A$. See \autoref{ex:EM} below for more details. 
\end{ex}

Before we end this section, we observe that our construction of the $K$-theory spectrum of a {\catcov} is a special case of a general technique.  This observation, along with its equivariant analogue in \autoref{G_spectrification}, will allow us to compare different constructions of $K$-theory in \autoref{sec:main}. 

\begin{df}\label{spectrification}
  Let $F\colon \Catfam \to \Spc_*$ be a functor from {\catcovs} to pointed spaces.  We can use $F$ to construct a
  functor $X_F\colon \Catfam \to \Gamma\Spc$ given on a {\catcov} $\C$ by
  \[X \longmapsto F(X \sma \C).\]  Let $K_F$ denote the
  functor from $\Catfam$ to the category of spectra given by mapping $\C$ to the spectrum associated to this $\Gamma$-space.
\end{df}

In the language of this definition, the $K$-theory spectrum of $\C$ from \autoref{ktheory} is denoted
\[K(\C) = K_{|N_{\sbt}\W(-)|}(\C).\]

The construction of spectra $K_F$ in \autoref{spectrification} is natural, in that a natural transformation between two functors $F_1, F_2\colon\Catfam \to \Spc_*$ induces a map of spectra between $K_{F_1}$ and $K_{F_2}$, and a natural equivalence of such functors induces a levelwise equivalence of spectra. In \autoref{G_spectrification}, we generalize this construction to $G$-{\catcovs}. Then, in \autoref{sec:main}, we use various categorical models to produce such natural transformations and show these induce the natural assembly map of \autoref{mainthm}.

\section{Weak products}

In this section we recall the  construction of the weak product. Fix an ambient pointed category with products, colimits, and a zero object. The four examples we use are the categories $\Sp$, $\Spc_*$, $s\Set_*$, and $\pCat$. Note that the one-object category with only the identity morphism is a zero object in $\pCat$.
\begin{df}
	For any set of objects $\{X_\alpha\}_{\alpha \in A}$, their \emph{weak product} is defined as
	\[ \bigoplus_{\alpha \in A} X_\alpha := \underset{I \subseteq A \textup{ finite}}\colim\, \prod_{i \in I} X_i, \]
	the colimit of the finite products over the finite subsets $I \subseteq A$.
\end{df}

\begin{ex}\label{weak_product_spectra}
	In the category $\Sp$ of spectra, the weak product 
\[ \bigoplus_\alpha X_\alpha \subseteq \prod_\alpha X_\alpha \]
	is formed by taking at each spectrum level the subspace of the product in which only finitely many coordinates are not the basepoint. This colimit is a homotopy colimit of spaces by the usual argument in e.g.\ \cite[Lem 3.6]{strickland_cgwh}. The weak product in based spaces and in based simplicial sets is similar.
\end{ex}

\begin{ex}
  Consider a set of categories $\{\C_\alpha \in \pCat\}_{\alpha \in A}$. The
  weak product category
\[ \bigoplus_\alpha \C_\alpha \subseteq \prod_\alpha \C_\alpha \]
is the subcategory whose objects are tuples that are the basepoint object in
all but finitely many coordinates, and whose morphisms are tuples that are the
identity morphism of the basepoint object in all but finitely many
coordinates.  This is also an object in $\pCat$.
\end{ex}

\begin{lem}\label{weak_product_comm}
	The weak product commutes up to canonical isomorphism with both the nerve and the realization. In particular,
	\[
		\Bigl| N_{\sbt} \bigoplus_\alpha \C_\alpha\Bigr|
		\cong \Bigl| \bigoplus_\alpha N_{\sbt} \C_\alpha\Bigr|
		\cong \bigoplus_\alpha \bigl|N_{\sbt} \C_\alpha\bigr|.
	\]
\end{lem}
\begin{proof}
	The commutation with the nerve is simple: a $k$-tuple of composable arrows in $\bigoplus_\alpha \C_\alpha$ is, for each morphism in the tuple, nontrivial in only finitely many coordinates, and therefore across the entire $k$-tuple is only nontrivial in finitely many coordinates. It is therefore the same thing as an element of $\bigoplus_\alpha N_k \C_\alpha$.
	
	The realization commutes with the weak product because it commutes with both filtered colimits and finite products.
\end{proof}

\begin{lem}\label{wedge_to_weak_product}
For spectra $\{X_\alpha\}_{\alpha\in A}$, the inclusion of the wedge sum $\bigvee_\alpha X_\alpha$ into the weak product $\bigoplus_\alpha X_\alpha$ is a stable equivalence.
\end{lem}
\begin{proof}
	Denote spectrum level $k$ of $X_\alpha$ by $X_\alpha(k)$. At level $k$, the subspace
	\[ \bigoplus_\alpha X_\alpha(k) \subseteq \prod_\alpha X_\alpha(k) \]
	is the filtered colimit of the products over all the finite subsets of $A$, ordered by inclusion. As in \autoref{weak_product_spectra}, this colimit is a homotopy colimit of spaces. A homotopy colimit taken at each spectrum level defines a homotopy colimit of spectra, showing that $\bigoplus_\alpha X_\alpha$ is the filtered homotopy colimit of the finite products of the spectra $X_\alpha$.
	
	Filtered homotopy colimits give colimits on homotopy groups, and therefore the inclusion of each $X_\alpha$ into the weak product induces an isomorphism on stable homotopy groups
	\[  \bigoplus_\alpha \pi_*(X_\alpha) \xrightarrow{\ \cong\ } \pi_*\Big(\bigoplus_\alpha X_\alpha\Big).  \]
	The wedge sum $\bigvee_\alpha X_\alpha$ also has this property, so the inclusion of the wedge into the weak product is an isomorphism on homotopy groups. Hence the inclusion is a stable equivalence of spectra.
\end{proof}

 Let $X$ be a pointed set and $\C$ be a \catcov.  For each $x\in X^\circ$, let
  $\Pi_x$ denote the map of pointed sets $X\to \{x,*\}$ that preserves $x$ and collapses all other
  elements of $X$ to $*$.  The maps $\Pi_x$ induce a functor
\[  \Pi\colon \W(X\sma \C)\to \prod_{X^\circ} \W(\C).\]

\begin{lem}\label{weak_additivity}
The image of $\Pi$ lands in the weak product, and the induced functor
\[ \widetilde{\Pi}\colon \W(X\sma \C)\to \bigoplus_{X^\circ}\W(\C)\]
is an equivalence of categories. 
\end{lem}
This result and its proof are entirely analogous to \cite[Prop 2.11(3)]{inna-assem} in the case of assemblers.

\begin{proof}
Recall that $X^\circ$ is $X$ minus the basepoint.  For each $x\in X^\circ$, the map of based sets $\Pi_x\colon X\to \{x,*\}$ induces a functor of \catcovs
\[X\sma \C\cong \bigvee_{x\in X^\circ} \C\to \{x,*\}\sma \C\cong \C,\]
as in \autoref{wedgeex}.   Explicitly, this functor is the identity on the copy of $\C$ indexed by $x$ and collapses the remaining copies of $\C$ to the basepoint object $*\in \C$.  Applying the construction $\W$ of \autoref{w_functor}, we obtain functors $\W(\Pi_x)\colon \W(X\sma \C)\to \W(\C)$ for each $x$.   These form the coordinates of the functor 
\[ \Pi\colon \W(X\sma \C)\to \prod_{x\in X^\circ} \W(\C).\]
 Unpacking the constructions here, we see that an object in the source category is a finite set $I$ and a tuple of pairs $\{(x_i,A_i)\}_{i \in I}$ indexed by $I$, where $x_i \in X^\circ$ and $A_i \in \C$. The $x$th coordinate of the functor $\Pi$ restricts this tuple to those pairs in which $x_i = x$.   Since $I$ is finite, we see that $\Pi$ must factor through the weak product.

For each $x \in X^\circ$, let $I_x \subseteq I$ be the subset of those elements $i\in I$ for which  $x_i=x$. Then each object in the source category can be re-described as a collection of \emph{disjoint} sets $I_x$ indexed by $x \in X^\circ$, all but finitely many of which are empty, and objects $A_i \in \C$ for each $i \in I_x$.  In this notation, the $x$th coordinate of $\Pi$ sends an object in the source category to the set $I_x$ and the objects $A_i$ indexed by $I_x$.

	On the other hand, an object in the weak product category is a collection of sets $I_x$ for $x \in X^\circ$, all but finitely many of which are empty, and objects $A_i \in \C$ for each $i \in I_x$. The sets $I_x$ need not be disjoint. However, by applying bijections to these sets, we may replace this object with an isomorphic object where the $I_x$ are disjoint.  Define an object $\{(x_j,A_j)\}_{j\in J}$ in $\W(X\sma \C)$ by setting $J=\amalg_{x}I_x$ and letting $(x_j,A_j)=(x,A_i)$ when $j=i\in I_x$. By construction, $\widetilde{\Pi}$ sends this object to the chosen object in the weak product category; thus $\widetilde{\Pi}$ is essentially surjective.

	In the source, the morphisms are maps of finite sets $f\colon I \to J$ such that $x_i = x_{f(i)}$, together with covers $\{ A_i \to B_j \}_{i \in f^{-1}(j)}$ for $j \in J$. This is the same data as, for each $x \in X^\circ$, maps of sets $I_x \to J_x$ and covers $\{ A_i \to B_j \}_{i \in f^{-1}(j)}$ for $j \in J_x$. This shows that $\widetilde{\Pi}$ is also fully faithful.
\end{proof}

\begin{cor}
	The $\Gamma$-space defining $K(\C)$ in \autoref{ktheory} is special.
\end{cor}
\begin{proof}
The specialness condition for this $\Gamma$-space is the case of \autoref{weak_additivity} in which $X$ is a finite pointed set.
\end{proof}

Lastly, we give a description of  $\pi_0$ of the spectrum $K(\C)$. The proof of the following proposition is identical to that of \cite[Theorem 2.13]{inna-assem}, which is the analogous result about assemblers. 

\begin{prop}
The group $K_0(\C)$ is the free abelian group generated by the objects $[A]$, modulo the relation $[A]=\sum_{j\in J}[A_j]$ for every covering family $\{A_j\to A\}_{j\in J}$.
\end{prop}

\section{Homotopy orbit categories}

We next turn to {\catcovs} that have a $G$-action and their homotopy orbits. Let $G$ be a discrete group. We denote also by $G$ the category with one object, morphism set $G$, and composition $g \circ h = gh$.

\begin{df}
  A \emph{$G$-\catcov} is a functor $G \to \Catfam$.
\end{df}

Expanding the definition, this means that if $\C\in \Catfam$ is the image of the
object of $G$, then the group $G$ acts on $\C$ on the left through functors of {\catcovs}. By the naturality of the $K$-theory construction, such an action
induces a left action of $G$ on the spectrum $K(\C)$.

\begin{ex}\label{no_moving}
  If we take the trivial group in \autoref{polytopeex}, we get a category with
 covering families $\Pol{X}{1}$ whose objects are polytopes $P \subseteq X$ and
  whose morphisms are inclusions $P \subseteq Q$ as subsets of $X$. If $G$ is
  any subgroup of the isometry group of $X$, then each $g \in G$ acts on
  $\Pol{X}{1}$, sending each polytope $P$ to $gP$. This gives a left $G$-action
  on $\Pol{X}{1}$ in $\Catfam$; hence the spectrum $K(\Pol{X}{1})$
  inherits a $G$-action.
\end{ex}

\begin{df}\label{htpyorbitass}  
  Given a $G$-{\catcov} $\C$, we define the \emph{homotopy orbit \catcov} $\C_{hG}$
  to be the Grothendieck construction $G \int \C$. 
  This has the same objects as $\C$, but morphisms incorporate the $G$-action as follows.  For objects $X_0$ and $X_1$ where $X_1$ is not the basepoint, a morphism $X_1 \to X_0$ in $\C_{hG}$ is a pair $(f,g)$ of an element $g \in G$ and a
  morphism $f\colon gX_1 \to X_0$ in $\C$. The composition is given by
  \[ (f_1,g_1) \circ (f_2,g_2) = (f_1 \circ g_1(f_2),\ g_1g_2). \]
  When $X_1 = *$, a morphism $* \to X_0$ in $\C_{hG}$ is defined to be simply a morphism $* \to X_0$ in $\C$.   Composition of a morphism $f_2\colon * \to X_1$ and a morphism $(f_1,g_1)\colon X_1 \to
  X_0$ is given by $f_1 \circ g_1(f_2)$. 
By defining morphisms out of $\ast$ in this way, we are essentially ignoring the action of $G$ on $\ast$, which was required to be trivial anyway.  This is the natural construction in a pointed context; compare with the homotopy orbits of a group action on a pointed space. 

  We declare that a collection of morphisms
  \[ \{ (f_i,g_i)\colon A_i \to B \}_{i \in I} \]
  is a covering family in $\C_{hG}$ if the morphisms
  \[ \{ f_i\colon g_iA_i \to B \}_{i \in I} \]
  are a covering family in $\C$.
  This gives $\C_{hG}$ the structure of a \catcov.
\end{df}

\begin{lem}
  The homotopy orbit category $\C_{hG}$ is a {\catcov}.
\end{lem}

\begin{proof}
First note that $\{(1_A, 1)\colon A\xrightarrow{=} A\}$ is a cover since $\{1_A\colon A\xrightarrow{=} A\}$ is a cover in $\C$ by definition.
Given composable covering families
	\[ \{ (f_{ij},g_{ij})\colon A_{ij} \to B_i \}_{i \in I_j}, \quad \{(f_j,g_j)\colon B_j \to C \}_{j \in J} \]
	as in \autoref{family_structure}, the composite family $\{(f_j,g_j) \circ (f_{ij},g_{ij})\colon A_{ij} \to C \}_{j \in J, \ i \in I_j}$ can be rewritten as
	\[ \left\{(f_j \circ g_j(f_{ij}),g_jg_{ij})\colon A_{ij} \to C \right\}_{j \in J, \ i \in I_j}. \]
	For this to be a covering family in $\C_{hG}$, it suffices that
	\[ \{f_j \circ g_j(f_{ij})\colon g_jg_{ij}A_{ij} \to C \}_{j \in J, \ i \in I_j} \]
        is a covering family. This follows from the assumption that the action
        of each $g_j$ preserves covering families in $\C$, and that the covering
        families in $\C$ are preserved by composition.  An analogous argument
        works in the case of composition with morphisms from the basepoint, as
        well. 
\end{proof}

\begin{ex}\label{polytopeex2}
  Returning to \autoref{polytopeex}, it follows from the definitions that
  \[ (\Pol{X}{1})_{hG} \cong \Pol{X}{G}. \] In other words, the category of
  polytopes with action by $G$ is the homotopy orbits of the category of
  polytopes with no group action.
\end{ex}

\begin{ex}
  Let $G$ be a discrete group and let $1$ be the trivial group; let $G_*$ and $1_*$ be the corresponding {\catcovs} from \autoref{ex:groupascatwithoneobj}. By definition there is an isomorphism of \catcovs,
  \[ (1_*)_{hG} \cong G_*. \] As shown in \autoref{KtheoryBG}, $K(G_*)$ is equivalent to $\Sigma_+^\infty BG$. In particular,
  $K(1_*) \simeq \mathbb{S}$, and we have a commutation
  \[ (K(1_*))_{hG} \simeq K((1_*)_{hG}) = K(G_*). \]
  We can view this as a special case of \autoref{mainthm}.
\end{ex}

We point out that when $X$ is a pointed finite set (with no $G$-action), 
there are canonical isomorphisms
\[ X\sma (\C_{hG})\cong (X\sma \C)_{hG}. \]
As a result, the $\Gamma$-space that defines $K(\C_{hG})$ (see \autoref{ktheory})
is canonically isomorphic to the $\Gamma$-space
\[ X \mapsto \left| N_{\sbt}\W((X\sma\C)_{hG})\right|.\] 
Letting $F\colon G\Catfam\to \Spc_*$ be the functor $F(-)=\left|N_{\sbt}(\W(-)_{hG})\right|$, we deduce the following corollary.

\begin{cor}\label{twoGammaspaceversionofKforhtpyorbitsagree}
The spectrum $K_F(\C)$ associated to the functor $F(-)=|N_{\sbt}\W((-)_{hG})|$ is $K(\C_{hG})$.
\end{cor}

This is an example of the following more general construction, analogous to \autoref{spectrification}.
\begin{df}\label{G_spectrification}
  For any functor $F\colon G\Catfam \to \Spc_*$, we define $X_F\colon G\Catfam \to \Gamma\Spc$ by assigning the $G$-{\catcov} $\C$ to
  \[X \longmapsto F(X \sma \C). \]
We let $K_F(\C)$ be the spectrum associated to this $\Gamma$-space.
\end{df}

As in the nonequivariant case, natural transformations between functors $G\Catfam\to \Spc_*$ induce maps of associated spectra.  In the next section, we produce the assembly map of \autoref{mainthm} by defining a natural transformation with target the functor $\left|N_{\sbt}\W((-)_{hG})\right|$.

\section{Commutation of $K$-theory with homotopy orbits}\label{sec:main}

In this section we introduce a double category $\W^{\square}(\C_{hG})$ that arises naturally from the category with covers $\W(\C_{hG})$  when $\C$ is a $G$-\catcov.  This double category serves as the  comparison point for models for the homotopy orbit spectrum $K(\C)_{hG}$ and the spectrum $K(\C_{hG})$.   Showing that these models are equivalent allows us to prove \autoref{mainthm}.

Let $\C$ be a $G$-\catcov. Let $\C_{hG}$ be the homotopy orbit category from
\autoref{htpyorbitass} and $\W(\C_{hG})$ its category of covers from
\autoref{category_of_covers}.

\begin{df} We need two important sub-categories of $\W(\C_{hG})$.
  Both contain all objects.  We say that a morphism $f\colon \{A_i\}_{i\in I} \to
  \{B_j\}_{j\in J}$ is a 
  \begin{description}
  \item[covering sub-map] if each component $(f_i,g_i)\colon A_i \to B_{f(i)}$
    has $g_i = 1$, and a
  \item[move] if each component $(f_i,g_i)\colon A_i \to B_{f(i)}$ has $f_i =
    1_{A_i}$.  
  \end{description}
  Thus a covering sub-map decomposes an object into smaller objects but does not
  ``move them around,'' while a move permutes the components of an object and
  moves them around, but does not decompose the object.

  We denote covering sub-maps by
$\csm$
 and moves by
$\movemap.$
  When labeling covering sub-maps and moves by their
  components we omit the coordinates that are required to equal an identity.
\end{df}

The terminology is chosen to be compatible with the terminology in
\cite[Definition 1.5]{innak1}.    
  \begin{df}\label{double_cat}
 Consider the double category $\W^{\square}(\C_{hG})$ with
  \begin{itemize}
  \item the same objects as $\W(\C_{hG})$,
  \item horizontal morphisms the moves,
  \item vertical morphisms the covering sub-maps,
  \item and 2-cells  commuting squares in  $\W(\C_{hG})$. 
  \end{itemize} 
   \end{df} 
   Observe that this double category separates the group action and the covering maps, in that the horizontal morphisms are purely given by the action of group elements and the vertical morphisms are morphisms in the category of covers $\W(\C)$.  This separation will allow us to relate the classical construction of the homotopy orbits to the nerve of $\W^\square(\C_{hG})$ in \autoref{equivalence_one}.

To prove \autoref{mainthm}, we first show that the nerve of the double category $\W^\square(\C_{hG})$ is equivalent to the nerve of $\W(\C_{hG})$, naturally in the $G$-{\catcov} $\C$.  Thus we can model the spectrum $K(\C_{hG})$ via the $\Gamma$-space associated to the functor $\left|N_{\sbt\,\sbt}\W^\square((-)_{hG})\right|$.  The following  technical observation about the role of
   covering sub-maps and moves inside $\W(\C_{hG})$ is essential in this comparison.

\begin{lem} \label{lem:factor}
  Every morphism in $\W(\C_{hG})$ factors uniquely (up to unique isomorphism) as
  a move followed by a covering sub-map.
\end{lem}
\begin{proof}
  Consider a morphism $\{A_i\}_{i\in I} \to \{B_j\}_{j\in J}$ with components
  $(f_i,g_i)\colon A_i \to B_{f(i)}$.  Each component map factors uniquely as
  $(f_i, 1) \circ (1,g_i)$:  here $(f_i,1)$ is the type of component that can be in a
  covering sub-map and $(1,g_i)$ is the type of component that can be in a
  move.  Thus $f$ factors as
\[  \{B_j\}_{j\in J}  \stackrel{f_i}{\csmrev} 
\{g_iA_i\}_{i\in I}    \stackrel{g_i}{\movemaprev} \{A_i\}_{i\in I},
\]
  and the factorization is unique up to a permutation of the indexing set.
\end{proof}

As a corollary we have the following observation about completing 2-cells in
$\W^\square(\C_{hG})$:    
   \begin{lem} \label{lem:G-ind-bottom}
   Each diagram in $\W^{\square}(\C_{hG})$ of the form
\[\begin{tikzcd}[column sep={2em}, row sep={2em},]
 & \bullet\ar[d,no head, covsubmap ]\\
\bullet & \ar[l,>->]\bullet
\end{tikzcd}
\qquad\qquad \text{or} \qquad\qquad
\begin{tikzcd}[column sep={2em}, row sep={2em}]
\bullet \ar[d, no head, covsubmap]\\
\bullet & \ar[l, >->]\bullet
\end{tikzcd}
\]
  completes uniquely (up to unique isomorphism) to a 2-cell  in $\W^{\square}(\C_{hG})$.
   \end{lem}

   \begin{proof}
     We focus on the diagram on the left; the one on the right follows
     immediately from the observation that all moves are isomorphisms.
     Considering the diagram as a composition in $\W(\C_{hG})$, we see from
     Lemma~\ref{lem:factor} that we can factor the composite arrow uniquely as a
     move followed by a covering sub-map.  Since this factorization is unique
     (up to unique isomorphism) there is a unique 2-cell that completes it.
   \end{proof}
   
   \begin{rmk}
     We emphasize that \autoref{lem:factor} does not hold for factoring a
     morphism as a covering sub-map followed by a move, as it is possible that
     we move two ``pieces'' of an object by different elements of $G$, which we
     cannot do after they are ``glued together.''  Thus
     \autoref{lem:G-ind-bottom} does not hold for the analogous two squares with
     one of the given horizontal arrows at the top instead of the bottom. 
   \end{rmk}

\begin{notn}
  In the remainder of the paper we will have many morphisms whose names have indices and whose
  components we also want to index.  In order to avoid visual clutter, we
  will write, for example, $f_{2|i}$ instead of $(f_2)_i$ for the $i$-th
  component of the map $f_2$.
\end{notn}

 Using \autoref{lem:G-ind-bottom} we can show that $|N_{\sbt\,\sbt}\W^\square(\C_{hG})|$
   is a double-categorical model for $|N_{\sbt} \W(\C_{hG})|$. The double nerve $N_{\sbt\, \sbt\, }\W^{\square}(\C_{hG})$ is the
   bisimplicial set where a $(p,q)$-simplex is a $p\times q$ array of
   2-cells. We write these arrays as follows, with the horizontal morphisms
   pointing from right to left.   
   \begin{equation}\label{altgrid} 
\begin{tikzcd}[row sep=2.5em, column sep=2.5em]
     \{ g_{1|i} \cdots g_{q|i} P_i \}_{i \in I_p} \ar[d,"g_{1|i} \cdots g_{q|i}(f_p)",covsubmap] & \cdots  \ar[l,movemapstyle, "{g_{1|i}}"'] &
      \{ g_{q|i} P_i \}_{i \in I_p} \ar[d,"{g_{q|i}(f_p)}", covsubmap] \ar[l,"{g_{2|i}}"', movemapstyle]
       &
      \{ P_i \}_{i \in I_p} \ar[l, movemapstyle,"{g_{q|i}}"'] \ar[d,covsubmap,"{f_p}"]
      \\
      \vdots \ar[d, covsubmap, "{g_{1|i} \cdots g_{q|i}(f_2)}"] & &
      \vdots \ar[d,covsubmap, "{g_{q|i}(f_2)}"] &
      \vdots \ar[d, covsubmap, "{f_2}"]
      \\      
      \{ g_{1|i} \cdots g_{q|i} P_i \}_{i \in I_1} \ar[d,covsubmap,"{g_{1|i} \cdots g_{q|i}(f_1)}"] &
      \cdots       \ar[l, movemapstyle, "{g_{1|i}}"']  &
      \{ g_{q|i}P_i \}_{i \in I_1}  \ar[l,movemapstyle,"{g_{2|i}}"']
      \ar[d, covsubmap,"{g_{q|i}(f_1)}"]&
      \{ P_i \}_{i \in I_1} \ar[l, movemapstyle,"{g_{q|i}}"'] \ar[d, covsubmap,"{f_1}"]
      \\
      \{ g_{1|i}\cdots g_{q|i}P_i \}_{i \in I_0} &
      \cdots  \ar[l, movemapstyle, "{g_{1|i}}"'] &
      \{ g_{q|i}P_i \}_{i \in I_0} \ar[l, movemapstyle,"{g_{2|i}}"'] &
      \{ P_i \}_{i \in I_0} \ar[l,movemapstyle,"{g_{q|i}}"']
\end{tikzcd}
\end{equation}

   Note that in the $k$th row of horizontal morphisms, the indexing set $I_k$ is constant and the group elements  $g_{\ell|i}$ in this row are indexed  by $i \in I_k$. Note also that we have an agreement of group elements $g_{\ell|i} = g_{\ell|j}$ any time $j = f_k(i)$, since the squares commute in $\W(\C_{hG})$. As a result, the group elements in this diagram are determined by the ones lying on the bottom row.

 \begin{prop}\label{equivalence_three}
  There is a weak
  equivalence of spaces 
  \[|N_{\sbt\, \sbt\, }\W^{\square}(\C_{hG})| \longrightarrow |N_{\sbt} \W(\C_{hG})|\]
that is natural in the $G$-{\catcov} $\C$.
\end{prop}

\begin{proof}
Let $\W^{\boxbox}(\C_{hG})$ be the double category defined just as in \autoref{double_cat}, except the vertical maps are allowed to be any map in $\W(\C_{hG})$. We claim that we have a composite of homotopy equivalences after geometric realization
\[ N_{\sbt\, \sbt\, }\W^{\square}(\C_{hG})  \xrightarrow{\ \fbox{\small{1}}\ }    N_{\sbt\, \sbt\, }\W^{\boxbox}(\C_{hG}) \xrightarrow{\ \fbox{\small{2}}\ } N_{\sbt} \W(\C_{hG}). \]
Here the map $\fbox{\small{1}}$ is the inclusion of bisimplicial sets. The map $\fbox{\small{2}}$ should be viewed as a map of simplicial sets where the source is the diagonal of the bisimplicial set $N_{\sbt\,\sbt}\W^{\boxbox}(\C_{hG})$. This map is defined at level $q$ by sending a $q\times q$ array of $2$-cells to the $q$ maps in $\W(\C_{hG})$ given by looking at the boundaries of the $q$ 2-cells along the diagonal of the array.
 This map $\fbox{\small{2}}$ is  simplicial, and is a left inverse to the map
\[ N_{\sbt\, \sbt\, }\W^{\boxbox}(\C_{hG})  \xleftarrow{\ \sim\ } N_{\sbt\, } \W(\C_{hG}) \]
that sets all horizontal maps to be the identity. This latter map is an equivalence by Waldhausen's swallowing lemma \cite[Lemma 1.6.5]{1126}, hence its left inverse $\fbox{\small{2}}$ is an equivalence as well.

By the realization lemma, it is enough to show that the map $\fbox{\small{1}}$  of bisimplicial sets is an equivalence after fixing a direction. Let $\W_n^{\square}(\C_{hG})$ and  $\W_n^{\boxbox}(\C_{hG})$ be the categories whose objects are $n$-composites of vertical maps in $\W^{\square}(\C_{hG})$ and in $\W^{\boxbox}(\C_{hG})$, respectively, and whose morphisms are $n\times 1$ arrays of composable $2$-cells in the respective double categories. It suffices to show that for fixed $n$, the inclusion
\begin{equation}\label{eq:inclusion_one_level}
	\W_n^{\square}(\C_{hG}) \subseteq \W_n^{\boxbox}(\C_{hG})
\end{equation}
is an equivalence of categories. It is easy to see it is fully faithful---in both categories, the maps are defined by moves which make all of the resulting squares commute in the ambient category $\W(\C_{hG})$.

To see that \eqref{eq:inclusion_one_level} is essentially surjective, use
\autoref{lem:factor} to factor any morphism
\[ \{Q_j\}_{j\in J}\xleftarrow{(f_i, g_i)} \{P_i\}_{i\in I} \]
in $\W(\C_{hG})$ as
\[ \{Q_j\}_{j\in J} \stackrel{f_i}{\csmrev} \{g_i P_i\}_{i\in I} 
   \stackrel{g_i}{\movemaprev} \{P_i\}_{i\in I} .\]
Given any $n$-composite of maps in $\W(\C_{hG})$, we factor them each in this way and use \autoref{lem:G-ind-bottom} to complete the squares, which creates an isomorphism to an $n$-composite all of whose maps are of the form $(f_i, \id)$. For example, for $n=2$, essential surjectivity follows from the following diagram, where the dashed arrows are determined as in \autoref{lem:G-ind-bottom}.
\[\begin{tikzcd}
\{g_{1|i} g_{2|i} P_i \}_{i\in I_2} \ar[d,dashrightarrow,"{g_{1|i}(f_{2|i})}"'] &
\{g_{2|i} P_i \}_{i\in I_2} \ar[l,dashrightarrow, "{g_{1|i}}"'] \ar[d,covsubmap,"{f_{2|i}}"'] &
\{P_i\}_{i \in I_2} \ar[l,movemapstyle,"{g_{2|i}}"'] \ar[dl, rightarrow,"{(f_{2|i},g_{2|i})}"]
\\
\{g_{1|i} P_i\}_{i\in I_1} \ar[d, covsubmap,"{f_{1|i}}"'] &
\{P_i\}_{i\in I_1} \ar[dl,rightarrow,"{(f_{1|i}, g_{1|i})}"]  \ar[l,movemapstyle,"{g_{1|i}}"'] &
\\
\{P_i\}_{i\in I_0}
\end{tikzcd}\]
Viewing the composite along the diagonal as an object in $\W^\boxbox_2(\C_{hG})$, the horizontal maps provide an isomorphism to the object of $\W^\square_2(\C_{hG})$ given by the vertical composite on the left. This shows essential surjectivity.
\end{proof}

Recall that for a space or spectrum $E$ with an action of $G$, the homotopy orbits $E_{hG}$ can be defined as the realization of the simplicial object
\[ N_{\sbt}(S^0,G_+,E), \]
where the $q$th simplicial level is $(G)^{\sma q}_+\sma E = (G^q)_+ \sma E$. Taking $E$ to be $|N_{\sbt} \W(\C)|$, we get a simplicial space modeling the homotopy orbits $|N_{\sbt} \W(\C)|_{hG}$, which is the realization of a bisimplicial set $N_{\sbt}(S^0,G_+,N_{\sbt} \W(\C))$.  

Sending $\C$ to $\left|N_{\sbt} \W(\C)\right|_{hG}$  produces a functor $F\colon G\Catfam\to \Spc_*$, which,  as in \autoref{spectrification},  induces an associated spectrum.
 This is the standard model for the homotopy orbit spectrum $K(\C)_{hG}$.  We now turn to comparing this standard model with the double nerve $N_{\sbt\,\sbt}(\W^\square(\C_{hG}))$.

In proving the equivalence between $N_{\sbt\,\sbt}\W^\square(\C_{hG})$ and $N_{\sbt}\W(\C_{hG})$ of \autoref{equivalence_three}, we looked at the ``vertical'' slices  of the double category $\W^\square(\C_{hG})$.  In order to compare $N_{\sbt\,\sbt}\W(\C_{hG})$ with the usual construction of homotopy orbits, it is useful to consider the ``horizontal'' slices.

\begin{df}\label{horizontal_slices}
	For each $q \geq 0$, let $\W^q(\C_{hG})$
	denote the category whose objects are chains of $q$ composable horizontal morphisms in $\W^{\square}(\C_{hG})$, and whose morphisms are $1 \times q$ arrays of 2-cells.  
\end{df} 

\begin{lem}\label{same_as_old_paper}
	For each $q \geq 0$ there is a natural isomorphism of categories 
	\[ \W^q(\C_{hG}) \cong \W((G^q)_+ \sma \C), \]
where $G^q = G^{\times q}$ is the product of $q$ copies of $G$.
\end{lem}
In analyzing the right hand side, it is useful to recall that the category $(G^q)_+\sma \C$ is the wedge $\bigvee_{G^q} \C$, as mentioned in \autoref{wedgeex}.

\begin{proof}
For fixed $q$, objects of $\W((G^q)_+ \sma \C)$ are given by a choice of finite set $I$ and an $I$-tuple $\{ (g_{1|i},\dots,g_{q|i};P_i)\}_{i \in I}$ where each $P_i$ is an object of $\C$ and the element $(g_{1|i},\dots, g_{q|i})\in G^q$ indexes which ``copy'' of $\C$ the object $P_i$ lives in.  
Consider a morphism
\[ \{ (g_{1|i},\dots, g_{q|i};P_i) \}_{i\in I}\to \{(g_{1|j}, \dots, g_{q|j};P_j)\}_{j\in J}
\]
with associated map of sets $f\colon I\to J$. 
Because there are no morphisms in $(G^q)_+ \sma \C$ between non-basepoint objects living in different copies of $\C$, we can only have such a morphism if for each $i$,  $(g_{1|i},\dots,g_{q|i})=(g_{1|f(i)},\dots, g_{q|f(i)})$. In this case, for each $j\in J$, the maps $f_i$ simply make $\{P_i\}_{i\in f^\inv(j)}$ into a covering family of $P_j$.  

Now consider the objects of $\W^q(\C_{hG})$. For each $q$-tuple of horizontal
morphisms with indexing set $I$, the tuple of moves
    \[ \begin{tikzcd}
      \{ g_{1|i}\cdots g_{q|i}P_i \}_{i \in I} &
      \{ g_{2|i}\cdots g_{q|i}P_i \}_{i \in I} \ar[l,movemapstyle,"{g_{1|i}}"'] &
      \cdots \ar[l,movemapstyle,"{g_{2|i}}"'] &
      \{ g_{q|i}P_i \}_{i \in I} \ar[l,movemapstyle,"{g_{(q-1)|i}}"'] &
      \{ P_i \}_{i \in I} \ar[l,movemapstyle,"{g_{q|i}}"']
    \end{tikzcd}
  \]
     is determined by the object $\{ P_i \}_{i \in I}$, and the tuple of elements $\{(g_{1|i},\dots,g_{q|i})\}_{i \in I}$. Therefore the two categories in the statement have the same objects up to bijection. They have the same morphisms as well: in both cases the morphisms are covers that preserve the associated $q$-tuple of elements of $G$.
\end{proof}

These horizontal slices fit together to form $N_{\sbt\,\sbt}\W^\square(\C_{hG})$ in the following way.  As $q$ varies, the categories $\W^q(\C_{hG})$  form a simplicial category $\W^{\sbtup}(\C_{hG})$. The degeneracy maps are given by adding identity horizontal morphisms and identity 2-cells. The $0$th face map drops the moves $g_{1|i}$ and the last face map drops the moves $g_{q|i}$.  
The remaining face maps are given by composing horizontal morphisms.  The bisimplicial set $N_{\sbt\,\sbt}\W^\square(\C_{hG})$ is the levelwise nerve of the simplicial category $\W^{\sbtup}(\C_{hG})$.

The simplicial category structure on $\W^{\sbtup}(\C_{hG})$ becomes the following simplicial category structure on $\W((G^{\sbtup})_+ \sma \C)$ when we transport it along the isomorphism of \autoref{same_as_old_paper}.  On objects, the face maps are given by 
\begin{equation*}
d_k\bigl(\{(g_{1|i},\dots,g_{q|i};P_i)\}_{i\in I}\bigr)
=\begin{cases}
\{(g_{2|i},\dots,g_{q|i}; P_i)\}_{i\in I} & \text{if $k=0$}\\
\{(g_{1|i},\dots, g_{k|i}g_{k+1|i}, \dots, g_{q|i};  P_i)\}_{i\in I} & \text{if $0< k <q$}\\
\{(g_{1|i},\dots,g_{q-1|i};  g_{q|i}P_i)\}_{i\in I} & \text{if $k=q$}.
\end{cases}
\end{equation*}
Essentially, these are the face maps of the two-sided bar construction in the $I$-tuples of indexing variables.  Note that the last face map incorporates the action of $G$ on $\C$. Since all morphisms preserve the indexing variables, these formulas extend to give the face maps on morphisms in a straightforward way.  To check that the last face map extends to morphisms, observe that the action of $G$ on $\C$ preserves the property of being a covering family.

On objects, the degeneracy maps are
\begin{equation*} s_k\bigl( \{ (g_{1|i},\dots,g_{q|i}; P_i) \}_{i\in I}\bigr)
  = \{ (g_{1|i},\dots,g_{k|i},1,g_{k+1|i},\dots,g_{q|i}; P_i) \}_{i\in I}
  \end{equation*}
where $1$ is the identity element of $G$. Again, this formula extends to morphisms in the obvious way. We see that the double nerve $N_{\sbt\, \sbt\, }\W^{\square}(\C_{hG})$ is isomorphic to the bisimplicial set that we get by taking the nerve $N_{\sbt} \W((G^{\sbtup})_+ \sma \C)$ of this simplicial category.

In the classical homotopy orbits construction, at the $q$th simplicial level we have only a single $q$-tuple of elements of $G$.  In contrast, the data of $\W^q(\C_{hG})$ involves $q$-tuples of elements of $G$ that can vary over the indexing sets $I$.  To compare the two constructions, we make use of the following sub-double category of $\W^{\square}(\C_{hG})$, in which the group elements are ``uniform'' in the indexing sets $I$.  

\begin{df}\label{uniform}
	Let $\U^{\square}(\C_{hG}) \subseteq \W^{\square}(\C_{hG})$ be the sub-double category in which the horizontal morphisms are restricted to those where the elements $g_i$  are the same for all $i$. In the double nerve, this implies that in each column of horizontal morphisms (as depicted in (\ref{altgrid})), every group element $g_{k|i}$ in that column is equal. For each $q \geq 0$, let $\U^q(\C_{hG})$ be the corresponding category of horizontal $q$-tuples, as in \autoref{horizontal_slices}.
\end{df}

\begin{prop}\label{equivalence_one}
	There is an isomorphism of bisimplicial sets
	\[ N_{\sbt}(S^0,G_+,N_{\sbt} \W(\C)) \cong N_{\sbt\, \sbt\, }\U^{\square}(\C_{hG}) \]
that is natural in the $G$-{\catcov} $\C$.
\end{prop}

\begin{proof}
By definition, $N_q(S^0,G_+,N_{\sbt}\W(\C))$, the $q$th level of the left hand side,  is the nerve of the category \[(G^q)_+\sma \W(\C)=\bigvee_{G^q}\W(C).\]
That is, $N_q(S^0,G_+,N_{\sbt}\W(\C))$ is the wedge of the pointed categories $\W(\C)$ indexed by $G^q$. Note this wedge makes sense because there are no morphisms in $\W(\C)$ from non-basepoint objects to the basepoint object.  This wedge is the full subcategory of $\W((G^q)_+\sma\C)$ where the choice of indexing group elements does not vary in the indexing set $I$.

	Thus, for each $q \geq 0$, the isomorphism of \autoref{same_as_old_paper} restricts to an isomorphism
	\[ \U^q(\C_{hG}) \cong (G^q)_+ \sma \W(\C). \]
 Along this isomorphism, the simplicial structure described after \autoref{same_as_old_paper} becomes exactly the simplicial structure for the two-sided bar construction for the action of $G$ on the pointed category $\W(\C)$. Taking nerves, this gives the two-sided bar construction on $N_{\sbt} \W(\C)$, as desired.
\end{proof}

\begin{prop}\label{equivalence_two}
	The inclusion of bisimplicial sets, natural in $\C$,
	\[ N_{\sbt\, \sbt\, }\U^{\square}(\C_{hG}) \to N_{\sbt\, \sbt\, }\W^{\square}(\C_{hG}) \]
	induces a stable equivalence on the spectra associated to the $\Gamma$-spaces of these functors as in \autoref{G_spectrification}.
\end{prop}

\begin{proof}
	We work one level at a time in the horizontal nerve. Fixing $q$ and varying $k$, we need to show that the map along the top row of the following diagram produces a stable equivalence of spectra.
	\begin{equation}\label{compare_orbits_eq}
	\begin{tikzcd}[arrows=rightarrow] (G^q)_+ \sma \left|N_{\sbt} \mathcal W\left( S^k_{\sbt} \sma \C \right)\right|
	\ar[r,] \ar[d,"\sim"] &
	\left|N_{\sbt}\W\left( (G^q)_+ \sma S^k_{\sbt} \sma \C \right)\right|
	\ar[d,"\sim"] \\
	{\displaystyle\bigoplus_{G^q}} \left|N_{\sbt} \mathcal W\left( S^k_{\sbt} \sma \C \right)\right| \ar[r, "\cong"]
	& \Bigl|N_{\sbt} \Bigl({\displaystyle\bigoplus_{G^q}} \mathcal W\bigl( S^k_{\sbt} \sma \C \bigr)\Bigr)\Bigr|
      \end{tikzcd}
	\end{equation}
	Recall that the smashes with $G^q$ that appear in the top row of this diagram are constructed as wedge products in based simplicial sets and {\catcovs}, respectively:
\[ (G^q)_+ \sma \bigl| N_{\sbt}\W(S^k_{\sbt}\sma \C)\bigr|= \bigvee_{G^q} \bigl|N_{\sbt}\W(S^k_{\sbt}\sma \C)\bigr| \text{\quad and \quad} (G^q)_+ \sma S^k_{\sbt}\sma\C =\bigvee_{G^q}(S^k_{\sbt}\sma\C).\]
The vertical maps are both induced by the evident maps from the wedges to the products over $G^q$.  The right-hand vertical map lands in the weak product by \autoref{weak_additivity}. It is straightforward to see that the diagram commutes.
	
	The left vertical map is a stable equivalence by \autoref{wedge_to_weak_product}, and the right vertical map is a stable equivalence by \autoref{weak_additivity}. The bottom horizontal map is a homeomorphism by \autoref{weak_product_comm}. By the two-out-of-three property, the top horizontal map is therefore a stable equivalence as well. 
\end{proof}

These equivalences allow us to prove \autoref{mainthm}.

\begin{proof}[Proof of \autoref{mainthm}]
Combining Propositions \ref{equivalence_one}, \ref{equivalence_two}, and \ref{equivalence_three}, we get maps of functors $G\Catfam \to \Spc_*$
\[ \left|N_{\sbt}(S^0,G_+,N_{\sbt} \W(\C))\right| \stackrel{\cong}{\longrightarrow} \left|N_{\sbt\, \sbt\, } \U^\square(\C_{hG})\right| \longrightarrow \left|N_{\sbt\, \sbt\, } \W^\square(\C_{hG})\right| \stackrel{\sim}{\longrightarrow} \left|N_{\sbt} \W(\C_{hG})\right| \]
and thus induced maps on associated spectra, as in \autoref{G_spectrification}.  We refer to this composite as the \emph{assembly map}. The first and last maps are induced by equivalences of functors and are thus level equivalences of spectra.  The middle map induces a stable equivalence of spectra by \autoref{equivalence_two}.  Observe that by \autoref{twoGammaspaceversionofKforhtpyorbitsagree}, the $\Gamma$-space associated to the rightmost functor $G\Catfam\to \Spc_*$ produces the spectrum $K(\C_{hG})$. Thus the assembly map is a stable equivalence
\[ \asm\colon K(\C)_{hG} \xrightarrow{\simeq} K(\C_{hG}), \]
proving \autoref{mainthm}.
\end{proof}

We now obtain \autoref{polyhedrathrm} as a corollary, using the identification of $\Pol{X}{G}$ as $(\Pol{X}{1})_{hG}$ from \autoref{polytopeex2}.
\begin{cor}[Farrell--Jones isomorphism for $K$-theory of polyhedra]
 Let $X$ be any neat geometry (Euclidean, hyperbolic, or spherical), and let $G$ be any subgroup of the isometry group of $X$. The assembly map 
\[ K(\Pol{X}{1})_{hG}\to K(\Pol{X}{G})\]
 is an equivalence.
\end{cor}

\begin{rmk}
Since we use this formula later on, we specify explicitly the map
\[\diag N_{\sbt\, \sbt\, }\W^{\square}(\C_{hG}) \to N_{\sbt} \W(\C_{hG})\]
from \autoref{equivalence_three}, which we call the \emph{simplicial assembly
  map}. Note  that $N_{\sbt\, \sbt\, }\W^{\square}(\C_{hG})\cong N_{\sbt}
\W((G^{\sbtup})_+ \sma \C)$, as in \autoref{same_as_old_paper}. At simplicial level
$p$ an element of $\diag N_{\sbt\,\sbt}\W^{\square}(\C_{hG})$ is a $p\times p$
array as in  Diagram \ref{altgrid}. This data is equivalent to an $I_0$-indexed set of
$p$-tuples of elements in $G$ representing the bottom row of horizontal maps in
the array, and a composite in $\W(\C)$ representing the covering sub-maps along the right-hand side of the array:
\begin{equation*}\label{element1}
	\bigl( \{ (g_{1|i},\ldots,g_{p|i}) \}_{i \in I_0},\ \{ P_i\}_{i \in I_0} \stackrel{ f_{1|i}}{\csmrev} \{ P_i \}_{i \in I_1} \stackrel{f_{2|i}}{\csmrev} \dotsb    \stackrel{f_{p|i}}{\csmrev} \{ P_i \}_{i \in I_p}  \bigr)
\end{equation*}
Under the assembly map, this goes to the composite of morphisms in $\W(\C_{hG})$ that we get by traversing the associated array in $N_{\sbt\, \sbt\, }\W^{\square}(\C_{hG})$ diagonally, namely
\begin{equation}\label{elementinhtpyorbit}
\resizebox{\textwidth}{!}{$
	\{ g_{1|i}\cdots g_{p|i}P_i \}_{i \in I_0} \xleftarrow{(g_{1|i}\cdots g_{p|i}(f_{1|i}),\ g_{1|i})} \{ g_{2|i}\cdots g_{p|i}P_i \}_{i \in I_1} \xleftarrow{(g_{2|i}\cdots g_{p|i}(f_{2|i}),\ g_{2|i})} \dotsb 
\xleftarrow{(g_{p|i}(f_{p|i}),\ g_{p|i})} \{ P_i \}_{i \in I_p}.$
}
 \end{equation}
\end{rmk}

\begin{rmk}
The inverse equivalences of categories $\W_n^{\boxbox}(\C_{hG}) \to \W_n^{\square}(\C_{hG})$ from the proof of \autoref{equivalence_three} do \emph{not} assemble into a map of simplicial categories. Therefore we do not get a simplicial inverse to the map $\fbox{\small{1}}$ in the proof of \autoref{equivalence_three}. 
It is possible to define an explicit homotopy inverse to the map in \autoref{equivalence_three} at the simplicial level, but only if we first subdivide $N_{\sbt} \W(\C_{hG})$. More precisely, we can define an inverse simplicial homotopy 
\[\sd_2 N_{\sbt} \W(\C_{hG})\to \diag N_{\sbt\, \sbt\, }\W^{\square}(\C_{hG})\]
from the B{\"o}kstedt--Hsiang--Madsen subdivision defined in \cite{bokstedt_hsiang_madsen}, as we now sketch.
 
 Recall that  $\sd_2 N_n \W(\C_{hG}) = N_{2n+1} \W(\C_{hG})$. The inverse homotopy, schematically, looks as follows when $n=2$:
\begin{center}
  \begin{tikzpicture}
    \matrix (m) [matrix of math nodes, row sep=1.5em, column sep=1.5em]
    {
      \bullet & \bullet & \bullet & \bullet & \bullet & \bullet \\
      \bullet & \bullet &  \bullet & \bullet & \bullet & \\
      \bullet & \bullet & \bullet & \bullet & & \\
      \bullet & \bullet & \bullet & & & \\
      \bullet & \bullet & & & & \\
      \bullet & & & & & \\
    };
    \draw[>->, color=red!60!black] (m-1-6.mid west) --  (m-1-5.mid east);
    \draw[>->, color=red!60!black] (m-1-5.mid west) --  (m-1-4.mid east);
    \draw[>->, color=red!60!black] (m-1-4.mid west) --  (m-1-3.mid east);
    \draw[>->, color=red!60!black] (m-1-3.mid west) --  (m-1-2.mid east);
    \draw[>->, color=red!60!black] (m-1-2.mid west) --  (m-1-1.mid east);
    \draw[>->, color=red!60!black] (m-2-5.mid west) --  (m-2-4.mid east);
    \draw[>->, color=red!60!black] (m-2-4.mid west) --  (m-2-3.mid east);
    \draw[>->, color=red!60!black] (m-2-3.mid west) --  (m-2-2.mid east);
    \draw[>->, color=red!60!black] (m-2-2.mid west) --  (m-2-1.mid east);
    \draw[>->, color=red!60!black] (m-3-4.mid west) --  (m-3-3.mid east);
    \draw[>->, color=red!60!black] (m-3-3.mid west) --  (m-3-2.mid east);
    \draw[>->, color=red!60!black] (m-3-2.mid west) --  (m-3-1.mid east);
    \draw[>->, color=red!60!black] (m-4-3.mid west) --  (m-4-2.mid east);
    \draw[>->, color=red!60!black] (m-4-2.mid west) --  (m-4-1.mid east);
    \draw[>->, color=red!60!black] (m-5-2.mid west) --  (m-5-1.mid east);
    \draw[densely dotted, ->] (m-1-1) -- (m-2-1);
    \draw[densely dotted, ->] (m-2-1) -- (m-3-1);
    \draw[densely dotted, ->] (m-3-1) -- (m-4-1);
    \draw[densely dotted, ->] (m-4-1) -- (m-5-1);
    \draw[densely dotted, ->] (m-5-1) -- (m-6-1);
    \draw[densely dotted, ->] (m-1-2) -- (m-2-2);
    \draw[densely dotted, ->] (m-2-2) -- (m-3-2);
    \draw[densely dotted, ->] (m-3-2) -- (m-4-2);
    \draw[densely dotted, ->] (m-4-2) -- (m-5-2);
    \draw[densely dotted, ->] (m-1-3) -- (m-2-3);
    \draw[densely dotted, ->] (m-2-3) -- (m-3-3);
    \draw[densely dotted, ->] (m-3-3) -- (m-4-3);
    \draw[densely dotted, ->] (m-1-4) -- (m-2-4);
    \draw[densely dotted, ->] (m-2-4) -- (m-3-4);
    \draw[densely dotted, ->] (m-1-5) -- (m-2-5);
    \draw[solid,->,color=blue!60!black] (m-1-6) -- (m-2-5);
    \draw[solid,->,color=blue!60!black] (m-2-5) -- (m-3-4);
    \draw[solid,->,color=blue!60!black] (m-3-4) -- (m-4-3);
    \draw[solid,->,color=blue!60!black] (m-4-3) -- (m-5-2);
    \draw[solid,->,color=blue!60!black] (m-5-2) -- (m-6-1);

    \draw[thick, color=green!50!black] (-2.9,0.1) rectangle (-0.1,2.9);
  \end{tikzpicture}
\end{center}
Here we start with a 2-simplex in the subdivision---that is, a 5-simplex in
$N_{\sbt} \W(\C_{hG})$---which is given as the blue arrows on the
diagonal. Just as in the proof of \autoref{equivalence_three}, we factor each of
these into a move  followed by a covering sub-map and complete the squares using \autoref{lem:G-ind-bottom}. The image of the 
2-simplex represented by the diagonal maps is the $2\times 2$ diagram in the green box. We do not give the lengthy details of showing that this does indeed specify a simplicial inverse equivalence since we do not use this explicit inverse. 
\end{rmk}

\section{A trace map to group homology}\label{sec:trace}

In this section we change gears and construct a trace map to group homology from the $K$-theory of a {\catcov}. We first define this trace by a direct, explicit formula, and use it to produce new classes in the higher $K$-groups of scissors congruence. In the next section we explain how this trace factors through the equivalence of \autoref{mainthm}.

Let $G$ be a discrete group. Recall that a $G$-module $A$ is an abelian group
with a left action of $G$ through homomorphisms. Equivalently, $A$ is a left
module over the group ring $\Z[G]$. 

\begin{df}\label{volume}
  Let $\C$ be a $G$-category with covering families. A \emph{measure} on $\C$ is a choice of $G$-module $A$ and any one of the following equivalent sets of data.
  \begin{itemize}
  	\item A $G$-equivariant functor
  \[ \mu\colon \C^\circ \to \mathcal{E}\!A \]
  preserving covering families.
  	\item A $G$-equivariant functor of {\catcovs}
  \[ \mu\colon \C \to \mathcal{E}\!A_* \]
  such that $\mu^{-1}(*) = \{*\}$.
  	\item A $G$-equivariant function $\mu\colon \ob\C^\circ \to A$
  satisfying
  \begin{equation}\label{cover_volume_condition}
    \mu(P) = \sum_{i \in I} \mu(P_i)
  \end{equation}
  for every covering family $\{ P_i \to P \}_{i\in I}$ in $\C$.
  	\item A map of $\Z[G]$-modules
  \[ \mu\colon K_0(\C) \to A. \]
  \end{itemize}
\end{df}

Many natural invariants of polytopes may be viewed as measures, as illustrated by the examples below.
\begin{ex}\label{boring_volume}
	Let $\Pol{X}{1}$ be the category with covering families of polytopes and inclusions from \autoref{polytopeex}. Let $G$ be any subgroup of the isometry group so that we can view $\Pol{X}{1}$ as a $G$-category with covering families as in \autoref{no_moving}. A natural measure on $\Pol{X}{1}$ is the function that takes each polytope to its volume. This takes values in $A=\R$, with trivial $G$-action.
      \end{ex}

\begin{ex}
	Consider polygons in two-dimensional Euclidean space, $P \subseteq E^2$, and let $G = T(2) \cong \R^2$ be the group of translations of the plane. For each one-dimensional subspace $V \subseteq \R^2$, the Hadwiger invariant along $V$ takes the signed sum of the lengths of the edges of $P$ parallel to $V$. The sign is picked according to which side of the edge points to the interior of $P$. This defines a measure on $\Pol{E^2}{1}$, taking values in $\R$, with trivial action by $T(2)$.
	See \cite{dupont_82,dupont_book} for more details.
\end{ex}

\begin{ex}
	For polyhedra in three-dimensional Euclidean space $P \subseteq E^3$, the Dehn invariant is an element of $\R\otimes_\Z \R/\pi\Z$ given sending by $P$ to the sum over its edges of the length of the edge tensored with its dihedral angle, giving an element of $\R \otimes_\Z \R/\pi\Z$ (see e.g.~\cite{dupont_book,inna-perspectives}). This defines a measure with respect to the group $E(3)$ of all isometries of $E^3$, where we view $E(3)$ as acting trivially on $\R\otimes_\Z\R/\pi\Z$.
\end{ex}

The zeroth algebraic $K$-group of a category with covering families $\C$ receives the universal measure on $\C$.
\begin{ex}\label{universal_volume}
  The \emph{universal measure} sends each object of $\C$ to its $K$-theory class in $A = K_0(\C)$. It is easy to observe that any other measure on $\C$ factors through this one.
\end{ex}

In the remainder of this section, we will explain how a measure $\mu\colon \C\to \mathcal{E}\!A_*$ gives rise to an explicit trace map from the $K$-groups $K_n(\C_{hG})$ to group homology $H_n(G;A)$.  This proves the existence of the claimed trace map in \autoref{intro_trace}.

Let $N_{\sbt}^\otimes$ refer to the two-sided bar construction with respect to the tensor product. We define a map of simplicial sets 
\[ T\colon N_{\sbt} \W(\C_{hG}) \to N_{\sbt}^\otimes (\Z,\Z[G],A) \]
by sending a $p$-tuple of maps
\[\{P_{0|i}\}_{i\in I_0}  \xleftarrow{(f_{1|i},g_{1|i})} \{P_{1|i}\}_{i\in I_1}  \xleftarrow{(f_{2|i},g_{2|i})} \cdots  \xleftarrow{(f_{p|i},g_{p|i})} \{P_{p|i}\}_{i\in I_p}
 \]
to the element of $\Z[G]^{\otimes p}\otimes A$ given by the sum
\begin{equation}\label{trace_formula}
	\sum_{i\in I_p} g_{1|{f_2\cdots f_p(i)}}
	\otimes g_{2|{f_3\cdots f_p(i)}}
	\otimes \cdots
	\otimes g_{{p-1}|{f_p(i)}}
	\otimes g_{p|i}
	\otimes \mu(P_{p|i}).
\end{equation}
This map records the string of symmetries from $G$ that are applied to a piece $P_{p|i}$ in the ``finest'' cover $\{P_{p|i}\}_{i\in I_p}$ as we trace through its images under the string of covering maps, together with the value of the measure on $P_{p|i}$. We usually simplify notation by writing $g_{k|i}$ in the place of $g_{k|f_{k+1} \cdots f_p(i)}$ when $i \in I_p$, so that this formula becomes
\begin{equation*}\label{trace_formula_simplified}
	\sum_{i\in I_p} g_{1|i}
	\otimes g_{2|i}
	\otimes \cdots
	\otimes g_{{p-1}|i}
	\otimes g_{p|i}
	\otimes \mu(P_{p|i}).
\end{equation*}

\begin{lem}\label{trace_simplicial}
	The map $T\colon N_{\sbt}\W(\C_{hG})\to N^\otimes_{\sbt}(\Z,\Z[G],A)$ is simplicial.
\end{lem}

\begin{proof}
The condition \eqref{cover_volume_condition} is precisely what we need for $T$ to commute with the last face map.  Note that $I_p$ could be empty and $I_{p-1}$ nonempty,  in which case \eqref{cover_volume_condition} implies that the objects  $\{P_{p-1}^i\}_{i\in I_{p-1}}$ have zero measure, so that the last face map still commutes with $T$. It is straightforward to check that $T$ commutes with the other face and degeneracy maps.
\end{proof}

We next generalize $T$ to a map of simplicial sets
\begin{equation}\label{T_X}
	T_X\colon N_{\sbt} \mathcal W(X \sma \C_{hG}) \to X \otimes N_{\sbt}^\otimes (\Z,\Z[G],A)
\end{equation}
for each based set $X$. The smash product $X \sma \C$ is defined in \autoref{wedgeex}. For abelian groups $M$, we define $X \otimes M$ similarly as a direct sum of copies of $M$ over $X$ minus its basepoint:
\[ X \otimes M := \bigoplus_{X^\circ} M. \]
This also forms a functor in the based set $X$ in an obvious way (see \autoref{ex:EM}).

We define $T_X$ by the same formula as before, except that each $P_{p|i}$ has an associated point $x \in X$, and we send it to the corresponding summand of $X \otimes (\Z[G]^{\otimes p} \otimes A)$ indexed by $x$. 
\begin{lem}
	The map $T_X$ is simplicial and natural in $X$.
\end{lem}

\begin{proof}
	The proof that $T_X$ is simplicial amounts to applying \autoref{trace_simplicial} for each point of $X$ separately. Naturality is clear for injective maps and maps that fold points to the basepoint. For surjective maps, the point is that the formula \eqref{trace_formula} sends disjoint unions of composable covers to sums. (Conceptually, this says that $T$ is a monoid homomorphism, and should therefore pass to group completions.)
\end{proof}

Letting $X$ vary over all finite based sets, the realizations $|T_X|$ therefore define a map of special $\Gamma$-spaces, and therefore a map of symmetric spectra. The source is $K(\C_{hG})$ by \autoref{ktheory}. The target is the iterated bar construction on the topological abelian group $|N_{\sbt}^\otimes (\Z,\Z[G],A)|$. This is a classical model for the homotopy orbits spectrum $(HA)_{hG}$. By the Dold--Kan correspondence, the homotopy groups of this spectrum are the homology of the chain complex $C_{\sbt}(G;A)$ that computes group homology.

\begin{df}\label{concrete_trace}
Suppose we have a $G$-category with covering families $\C$ and a measure $\mu\colon \C^\circ \to \mathcal{E}\!A$. The {\bf explicit trace} is the map of spectra
\[  K(\C_{hG}) \xrightarrow{\ \tr\ }  (HA)_{hG}  \]
induced by $T_{(-)}$, or equivalently the map on homotopy groups
\[  K_n(\C_{hG}) \xrightarrow{\ \tr\ }  H_n(G;A).  \]
These are equivalent because the target spectrum is a product of Eilenberg--MacLane spectra.
\end{df}
This explicit trace defines the map to group homology claimed in \autoref{intro_trace}; to complete the proof of this theorem, we relate the explicit trace to the assembly map in the next section.

The main case of interest for us is when $\C = \Pol{X}{G}$, that is, polytopes in any neat geometry $X$ up to any subgroup $G$ of the isometry group, from \autoref{polytopeex}. As in \autoref{polytopeex2}, we view $\Pol{X}{G}$ as the homotopy orbits $(\Pol{X}{1})_{hG}$. Taking $\mu$ to be the volume as in \autoref{boring_volume}, this gives a trace map
\[ K_i(\Pol{X}{G}) \longrightarrow H_i(G;\R). \]

\begin{ex}\label{volume_trace_easy}
	At $K_0$, this trace takes the class of each polytope $P$ to its volume $\mu(P)$ as an element of $H_0(G;\R) \cong \R$. This can be read off from \eqref{trace_formula}, at simplicial level 0.
\end{ex}

From \cite{innak1}, any operation that decomposes a polytope $P$ into pieces, then rearranges the pieces by elements of $G$ to form the same polytope $P$, gives a class in $K_1(\Pol{X}{G})$. We call such operations \emph{scissors automorphisms}. The image of each scissors automorphism in $H_1(G;\R) \cong G^\ab \otimes_\Z \R$ can be computed by applying \eqref{trace_formula} and then passing to homology. 

\begin{ex}\label{line_trace_1}
	Consider one-dimensional Euclidean geometry $E^1$. Let $G = T(1) \cong \R$ be the translation group, and $\mu$ the volume---i.e.~length---of the polytopes  as in \autoref{boring_volume}. Let $I_x$ be any interval of length $x$. Consider the two maps in $\Pol{E^1}{T(1)}$
	\[ \begin{tikzcd}
		I_x \amalg I_y \ar[->,bend right]{r}\ar[->,bend left]{r} & I_{x+y}
\end{tikzcd}
\]
	that glue the intervals end-to-end in opposite orders, as in \autoref{fig:interval_exchange}.
\begin{figure}[h]
	\centering
	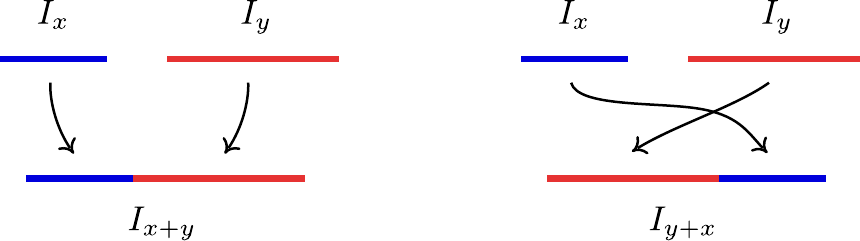
	\caption{Two ways to cover $I_{x+y}$ by the formal disjoint union $I_x \amalg I_y$.}\label{fig:interval_exchange}
\end{figure}

Without loss of generality, we may locate $I_x$ and $I_y$ inside $I_{x+y}$ so that the first of these covers is the identity map, and the second cover swaps the two intervals past each other. These form a loop in the classifying space $B\W(\Pol{E^1}{T(1)})$. Under $T$, the identity map 1-simplex goes to the chain
	\[ [0] \otimes x + [0] \otimes y \in C_1(T(1);\R) \cong \Z[T(1)] \otimes_\Z \R, \]
	and the swap 1-simplex goes to the chain 
	\[ [y] \otimes x + [-x] \otimes y. \]
	Together these form a cycle, whose image in homology is 
	\[ y \otimes x - x \otimes y  \in H_1(T(1);\R) \cong \R \otimes_\Z \R. \]	
	By \cite[Section 4]{innak1} we have the isomorphism $K_1(\Pol{E^1}{T(1)})\cong \R\wedge_\Q\R$, with $x\sma y$ represented by the loop described just above.  Therefore the explicit trace map is given by the formula
	\[ \begin{array}{rcccl}
		K_1(\Pol{E^1}{T(1)}) \cong &\R \wedge_\Q \R &\longrightarrow& \R \otimes_\Q \R &\cong H_1(T(1);\R) \\
		&x \wedge y &\longmapsto& y \otimes x - x \otimes y. &
	\end{array} \]
\end{ex}

\begin{ex}\label{rotation_example}
	Take the geometry $E^2$, and take $G = SE(2)$ to be the orientation-preserving isometries. Define a scissors automorphism of a square by cutting along the inscribed red square shown in \autoref{fig:rotated_rectangle}, where $\theta$ is an irrational multiple of $\pi$. Move the red square off to the side, and rotate it by an angle of $\theta$ to make it parallel to the sides of the original black square. Then, using only cuts and translations, rearrange this square and the remaining four right triangles back into the original rectangle. This is possible by Hadwiger's calculation of $K_0(\Pol{E^2}{T(2)})$; see \cite{dupont_82,dupont_book}.  The process is illustrated in \autoref{fig:reassemble_square}. 
	
	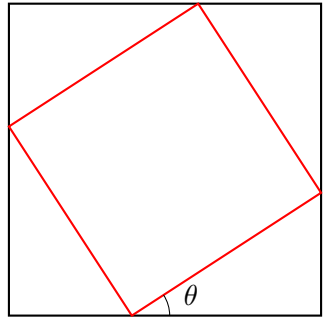
\begin{figure}[h]
	\centering
        \begin{tikzpicture}[scale=3]
          \def\th{33}
              \coordinate (A) at (0, {cos(\th)});
              \coordinate (B) at ({sin(\th)}, 0);
              \coordinate (C) at ({cos(\th)+sin(\th)},{sin(\th)});
              \coordinate (D) at ({cos(\th)},{cos(\th)+sin(\th)});

              \draw[thick] (0,0) rectangle ({cos(\th)+sin(\th)},{cos(\th)+sin(\th)});
              \draw[thick,red]  (A) -- (B) -- (C) -- (D) -- cycle;
              \draw ({sin(\th)+.2*cos(\th)},0) arc (0:\th:{.2*cos(\th)}) node[label=right:$\theta$]{};
        \end{tikzpicture}
	\caption{A square cut into a square and four right triangles.}\label{fig:rotated_rectangle}
      \end{figure}

      \begin{figure}[h]
       
\begin{tikzpicture}[scale=0.75]
    \def\th{33}

    \def\shapea#1{\fill[purple] (#1) -- +({1-cos(\th)},{-(1-cos(\th))*cot(\th)}) -- +({1-cos(\th)}, {-(1+sin(\th)-cos(\th))/2}) -- cycle;}
    \def\shapeb#1{\fill[blue!40!white] (#1) -- ({sin(\th)+cos(\th)-1},0) -- (0,0) -- cycle;}
    \def\shapec#1{\fill[yellow!80!white] (#1) -- ++({1-cos(\th)},0) -- ++(0,{-(1-cos(\th))*cot(\th)}) -- cycle;}
    \def\shaped#1{\fill[green!40!white] (#1)  -- +({cos(\th)-1},0) -- +({1-cos(\th)-sin(\th)},{(3*cos(\th)-1-sin(\th))/2}) -- +({1-cos(\th)-sin(\th)},{cot(\th)*(cos(\th)+sin(\th)-1)}) -- cycle;}
    \def\shapee#1{\fill[red!40!white] (#1) -- +({1-cos(\th)-sin(\th)},{cot(\th)*(cos(\th)+sin(\th)-1)}) -- +({1-sin(\th)-cos(\th)},{cos(\th)}) -- ++({(1-sin(\th)-cos(\th))/2}, {cos(\th)}) -- ++(0, {(1-sin(\th)-cos(\th))/2}) -- ++({0-(1-sin(\th)-cos(\th))/2},0) -- cycle;}
    \def\shapef#1{\fill[cyan] (#1) rectangle +({(1-cos(\th)-sin(\th))/2},{(1-cos(\th)-sin(\th))/2});}
    \def\shapeg#1{\fill[red!80!black] (#1) -- +({cos(\th)},0) -- +({cos(\th)},{cos(\th)+sin(\th)-1}) -- +({cot(\th)*(cos(\th)+sin(\th)-1)},{cos(\th)+sin(\th)-1}) -- cycle;}
    \def\shapeh#1{\fill[yellow!80!red] (#1) -- ++(0,{cos(\th)-1}) -- ++({cot(\th)*(cos(\th)-1)},0) -- cycle;}
    \def\shapei#1{\fill[magenta] (#1) -- +( {cot(\th)*(cos(\th)-1)}, {cos(\th)-1}) -- +({(cos(\th)-sin(\th)-1)/2}, {cos(\th)-1}) -- cycle;}
    \def\shapej#1{\fill[green!80!black] (#1) -- +({cot(\th)*(cos(\th)+sin(\th)-1)},{cos(\th)+sin(\th)-1}) -- +({(3*cos(\th)-1-sin(\th))/2}, {cos(\th)+sin(\th)-1}) -- +(0,{1-cos(\th)});}
    \def\shapek#1{\fill[blue!80!black] (#1)  -- ++( {0-cos(\th)}, 0) -- ++(0, {1-sin(\th)-cos(\th)}) -- cycle;}

    \begin{scope}[scale=5]
    \coordinate (A) at (0, {cos(\th)});
    \coordinate (B) at ({sin(\th)}, 0);
    \coordinate (C) at ({cos(\th)+sin(\th)},{sin(\th)});
    \coordinate (D) at ({cos(\th)},{cos(\th)+sin(\th)});

    \shapea{A}
    \shapeb{A}
    \shapec{D}
    \shaped{B}
    \shapee{C}
    \shapef{{cos(\th)+sin(\th)},{cos(\th)+sin(\th)}}
    \shapeg{B}
    \shapeh{C}
    \shapei{D}
    \shapej{A}
    \shapek{D}
 
    \draw[thick] (0,0) rectangle ({cos(\th)+sin(\th)},{cos(\th)+sin(\th)});
    \draw[thick]  (A) -- (B) -- (C) -- (D) -- cycle;

    \draw[|-|,yshift=-1pt] (0,0) to node[midway,below] {$\sin\theta$} ({sin(\th)},0);
    \draw[|-|,yshift=1pt] (0,{sin(\th)+cos(\th)}) to node[midway,above] {$\cos\theta$} ({cos(\th)},{sin(\th)+cos(\th)});
    \def\ct{{cos(\th)}}
    \def\st{{sin(\th)}}
    \path[shift={({-0.06*sin(\th)},{0.06*cos(\th)})}] ({sin(\th)},0) to node[midway] {$1$} ({sin(\th)+cos(\th)},{sin(\th)});
    \end{scope}

    \begin{scope}[xshift=25em, scale=5]
    \shapea{{2*cos(\th)+sin(\th)-2},{(sin(\th)-cos(\th)+1)/2}}
    \shapeb{0,{cos(\th)}}
    \shapec{{2*cos(\th)+sin(\th)-2},{(sin(\th)-cos(\th)+1)/2}}
    \shaped{{cos(\th)+sin(\th)-1}, {(1+sin(\th)-cos(\th))/2}}
    \shapee{{cos(\th)+sin(\th)-1},{(1+sin(\th)-cos(\th))/2}}
    \shapef{{(cos(\th)+sin(\th)-1)/2},{cos(\th)+sin(\th)}}
    \shapeg{{sin(\th)},1}
    \shapeh{{sin(\th)},{2-cos(\th)}}
    \shapei{{sin(\th)},{2-cos(\th)}}
    \shapej{{sin(\th)},1}
    \shapek{{(3*cos(\th)+sin(\th)-1)/2},{cos(\th)+sin(\th)}}
    
    \draw[thick] (0,0) rectangle ({cos(\th)+sin(\th)},{cos(\th)+sin(\th)});
    \draw[thick] ({cos(\th) + sin(\th)-1},0) -- ++(0,1) -- ++(1,0);
    \end{scope}

  \end{tikzpicture}
  \caption{A self-scissors congruence of a square with a $1\times 1$ subsquare
    rotated and all other pieces translated.}\label{fig:reassemble_square}
  \end{figure}
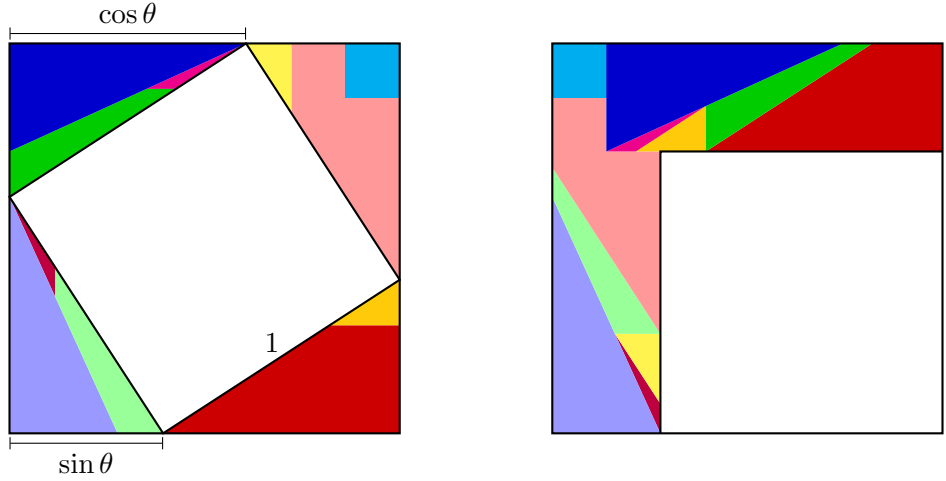
	The trace of this automorphism lies in
	\[ H_1(SE(2);\R) \cong SE(2)^\ab \otimes_\Z \R \cong \R/2\pi\Z \otimes_\Z \R. \]
	It is easy to see that its trace only depends on the first move that rotates the red rectangle---it is the only nonzero contribution to the $\R/2\pi\Z$ term. This contribution is (up to sign) $\theta \otimes x$, where $x$ is the area of the red rectangle. Since $\theta$ is an irrational multiple of $\pi$, the tensor $\theta \otimes x$ is nonzero. We have therefore proven the existence of a nonzero class in $K_1(\Pol{E^2}{SE(2)})$.
\end{ex}

In fact, by varying $\theta$ and $x$ in the previous example, we obtain the following result.
\begin{cor}
	The trace $K_1(\Pol{E^2}{SE(2)}) \to H_1(SE(2);\R)$ is surjective.
\end{cor}

\section{A second description of the trace}

In this section we explain how the trace of \autoref{concrete_trace} factors through the equivalence of \autoref{mainthm}, completing the proof of \autoref{intro_trace}.  We begin by further unraveling the {\catcov} $\mathcal{E}\!A_*$ from \autoref{ex:chaotic-group} and its $K$-theory. 
\begin{lem} \label{ex:EM}
  For any abelian group $A$, let $\mathcal{E}A_*$ be the {\catcov} from \autoref{ex:chaotic-group}. Then $K(\mathcal{E}A_*) \simeq HA$.
  \end{lem}
  
  \begin{proof}
  We construct an explicit equivalence $K(\mathcal{E}\!A_*) \simeq HA$ where $HA$ denotes the Eilenberg--MacLane spectrum on $A$.

  Recall that $HA$ may be constructed as the spectrum associated to the $\Gamma$-space
  \begin{equation}\label{EM0}
  	X \mapsto X \otimes A := \bigoplus_{x \in X^\circ} A,
  \end{equation}
  where as in \autoref{wedgeex}, $X^\circ$ refers to $X$ minus its basepoint. Maps of based sets $f\colon X \to Y$ act on $X \otimes A$ in the standard way, by deleting summands sent to the basepoint and applying $f$ to the rest.
  
  Let $UA$ be the underlying set of $A$, considered as a discrete category, with basepoint $0 \in A$.
  Consider the map of categories
  \begin{equation}\label{EM1}
  	\W( X \sma \mathcal{E}\!A_* ) \to U( X \otimes A ) := U\bigl( \bigoplus_{x \in X^\circ} A \bigr)
  \end{equation}
  defined by
  \[ \left\{(x_i,a_i)\right\}_{i\in I} \mapsto \left( \sum_{ x_i = x,\, i\in I} a_i \right)_{x \in X^\circ}. \]
  In particular, the empty tuple goes to $0 \in \bigoplus_{x \in X^\circ} A$, so that this is a map of pointed categories. It is easy to see this is well-defined on morphisms, since the morphisms in $\W( X \sma \mathcal{E}A_* )$ are covers that preserve the sums of the elements in $A$. Furthermore the construction is natural in $X$. Taking realizations, we get a map of special $\Gamma$-spaces
  \[ \left| N_{\sbt} \W\left( X \sma \mathcal{E}\!A_* \right) \right| \to X \otimes A \]
 and therefore a map on the associated spectra
  \begin{equation}\label{EM2}
  	K(\mathcal{E}\!A_*) \to HA.
  \end{equation}
  Observe that for finite $X$, the preimage of any tuple $(a_x)_{x \in X^\circ} \in U(X \otimes A)$ under the map in \eqref{EM1} is a
  connected component of $\W( X \sma \mathcal{E}\!A_* )$ which has the tuple $\{(x,a_x)\}_{x \in X^\circ}$ as a terminal object.  Thus by Quillen's Theorem A,  the map of categories \eqref{EM1} induces an equivalence on classifying spaces.  Hence the map of spectra \eqref{EM2} is an equivalence as well.
\end{proof}

\begin{rmk}
	If $A$ is a $\Z[G]$-module, it is easy to see that $\mathcal{E}\!A_*$ is a $G$-{\catcov} and that the equivalence \eqref{EM2} is $G$-equivariant. Therefore it gives an equivalence of homotopy orbit spectra.
\end{rmk}

Now consider any $G$-category with families structure $\C$, and let $\mu\colon \C^\circ \to \mathcal{E}\!A$ be a measure taking values in the $\Z[G]$-module $A$.

\begin{df}\label{abstract_trace}
The {\bf abstract trace}
  $ K(\C_{hG}) \to (HA)_{hG} $
  is defined as either composite in the following commutative diagram  in the stable homotopy category.
  \[ \begin{tikzcd}[arrows=rightarrow]
  	K(\C_{hG}) \ar[r,"\mu"] & K((\mathcal{E}A_*)_{hG})  \\
  	K(\C)_{hG} \ar[r,"\mu"]\ar[u,"\sim"] & K(\mathcal{E}A_*)_{hG}\ar[u,"\sim"'] \ar[r,"\sim"] & (HA)_{hG}.
 \end{tikzcd} \]
  The horizontal maps in the square are induced by $\mu$. The vertical maps are the equivalence from \autoref{mainthm}. The final map is induced by the equivariant equivalence from \autoref{ex:EM}.
\end{df}

\begin{thm}\label{traces_agree}
	The abstract trace of \autoref{abstract_trace} agrees in the homotopy category with the explicit trace of \autoref{concrete_trace}.
\end{thm}
This identification completes the proof of \autoref{intro_trace} of the introduction.

\begin{proof}
	The explicit trace is natural with respect to $G$-equivariant maps of {\catcovs} that commute with $\mu$. So without loss of generality $\C = \mathcal{E}A_*$ and $\mu$ is the identity.
	
	It suffices to show there is a commuting diagram of $\Gamma$-spaces of the form
	\[ \begin{tikzcd}[arrows=rightarrow]
\left|N_{\sbt}(S^0,G_+,N_{\sbt} \W(X \sma \mathcal{E}A_*))\right| \ar[r,"\sim"] \ar[d,"\sim"'] &
		\left|N_{\sbt} \W(X \sma (\mathcal{E}A_*)_{hG})\right| \ar[d,"{T_X}"]
		\\
		\left|N_{\sbt}(S^0,G_+,U(X \otimes A))\right| \ar[r,"\cong"] &
		\left|X \otimes N^\otimes_{\sbt}(\Z,\Z[G],A)\right|.
\end{tikzcd} \]
	Here the top map is the simplicial assembly map, given explicitly by the formula in \eqref{elementinhtpyorbit}. The left-hand map is the equivalence \eqref{EM1} from \autoref{ex:EM}, the bottom map is a canonical identification, and the right-hand map is the map $T_X$ from \eqref{T_X}.
	
	We check that this diagram commutes before realization, after passing to the diagonal of the bisimplicial set in the top-left. Given a $p$-tuple of group elements and composable maps of covers, both branches hold onto the $p$-tuples of group elements in $G$, together with the associated $x \in X$ and sum of elements in $A$ for each such tuple, and forget everything else.
\end{proof}

In fact, the trace map of \autoref{intro_trace} is determined by its action on $\pi_0$.  Thus, any map $K(\C)\to HA$ that applies the measure $\mu$ on $K_0(\C)$ can be used to compute the trace.  This flexibility is valuable in applications, including in \autoref{example:caryscoolpaperpolyhedra} below, where we identify the trace with the homotopy orbits of the truncation map from $K(\C)$ to its $0$th Postnikov stage $H(K_0(\C))$.

Recall that a spectrum is connective if its negative homotopy groups vanish, and that $HA$ is characterized by the fact that its homotopy is $A$ in degree zero and 0 in all other degrees. Recall further that the homotopy category of Borel $G$-spectra consists of spectra with $G$-action, up to equivariant maps that are stable equivalences after forgetting the $G$-action.
\begin{lem}\label{maps_to_em}
	If $A$ is a $\Z[G]$-module and $K$ is a connective Borel $G$-spectrum then maps in the Borel homotopy category
	\[ K \to HA \]
	are in bijection with $\Z[G]$-module homomorphisms $\pi_0 K \to A$.
\end{lem}

\begin{proof}
	This is a standard obstruction theory argument for Eilenberg--MacLane spectra. We may replace $K$ with a free $G$-cellular spectrum and $HA$ with a fibrant spectrum.  This allows us to calculate maps in the homotopy category via maps of spectra, up to homotopy.
	
	Any homomorphism on $\pi_0$ defines $K \to HA$ on the 0- and 1-cells, and the definition extends inductively to the higher cells because $\pi_k(HA) = 0$ for $k > 0$. Furthermore, any other map $K \to HA$ inducing the same homomorphism on $\pi_0$ is homotopic on the 0-cells by a direct inspection. This homotopy can then be extended to the higher cells, again because $\pi_k(HA) = 0$ for $k > 0$. Therefore any two such maps are homotopic.
\end{proof}

\begin{cor}
	The abstract trace is identified with the homotopy $G$-orbits of any map
	\[ K(\C) \to HA \]
	that on $\pi_0$ applies the measure $\mu\colon K_0(\C) \to A$.
\end{cor}

\begin{ex}\label{example:caryscoolpaperpolyhedra}
	Taking the universal measure of \autoref{universal_volume}, the trace gives a map
	\[ K_i(\C_{hG}) \to H_i(G;K_0(\C)) \]
	for any $G$-{\catcov} $\C$. By the abstract definition, this is the same as the homotopy orbits of the truncation map $K(\C) \to H(K_0(\C))$. By \cite{scissors_thom}, for $\C = \Pol{X}{1}$ this truncation map is a rational isomorphism, and therefore so is the trace.
	
	In \cite{scissors_thom}, the rational higher scissors congruence problem is reduced to a question about group homology. This example shows that the reduction is accomplished directly by the trace: taking the universal measure $A = K_0(\Pol{X}{1})$, the trace is a rational isomorphism
	\[ K_i(\Pol{X}{G}) \otimes \Q \cong H_i(G;K_0(\Pol{X}{1})) \otimes \Q. \]
\end{ex}

 \bibliographystyle{alpha}
  \bibliography{references}

\end{document}